\documentclass[10pt]{amsart}
\usepackage{amsfonts,amssymb}
\usepackage[all]{xy}

\newtheorem{thm}{Theorem}[section]
\newtheorem{lem}[thm]{Lemma}
\newtheorem{prop}[thm]{Proposition}
\newtheorem{defini}[thm]{Definition}
\newtheorem{cor}[thm]{Corollary}

\theoremstyle{remark}
\newtheorem{remark}[thm]{Remark}

\newcommand*\isom{\xrightarrow{\sim}}
\newcommand{\pair}[1]{\langle#1\rangle}
\newcommand{\largewedge}{\mbox{\Large $\wedge$}}

\def\qq{\mathbb{Q}}

\def\pp{\mathcal{P}}
\def\rr{\mathbb{R}}
\def\zz{\mathbb{Z}}

\def\CC{\mathbb{C}}

\def\mm{\mathcal{M}}
\def\ll{\mathcal{L}}

\def\cc{\mathcal{C}}
\def\jj{\mathcal{J}}
\def\ff{\mathcal{F}}
\def\oo{\mathcal{O}}

\def\dd{\mathcal{D}}

\def\hh{\mathcal{H}}

\def\tt{\mathcal{T}}

\def\ZZ{\mathcal{Z}}

\def\H{H}
\def\d{d}

\def\deldelbar{\partial \overline{\partial}}

\numberwithin{equation}{section}

\begin{document}

\title[Torus bundles and $2$-forms on the universal Riemann surface]{Torus bundles and $2$-forms on the universal family of Riemann surfaces}

\author{Robin de Jong}



\begin{abstract} We revisit three results due to Morita expressing certain natural integral cohomology classes on the universal family of Riemann surfaces $\cc_g$, coming from the parallel symplectic form on the universal jacobian, in terms of the Euler class $e$ and the Miller-Morita-Mumford class $e_1$. Our discussion will be on the level of the natural $2$-forms representing the relevant cohomology classes, and involves a comparison with other natural $2$-forms representing $e$, $e_1$ induced by the Arakelov metric on the relative tangent bundle of $\cc_g$ over $\mm_g$. A secondary object called $a_g$ occurs, which was discovered and studied by Kawazumi around 2008. We present alternative proofs of Kawazumi's (unpublished) results on the second variation of $a_g$ on $\mm_g$. Also we review some results that were previously obtained on the invariant $a_g$, with a focus on its connection with Faltings's delta-invariant and Hain-Reed's beta-invariant.
\end{abstract}

\maketitle
\thispagestyle{empty}
\tableofcontents

\section{Introduction} \label{intro}

Let $g \geq 2$ be an integer. Let $\mm_g$ denote the moduli space of compact Riemann surfaces of genus $g$ and denote by $p \colon \cc_g \to \mm_g$ the universal family of Riemann surfaces over $\mm_g$. The second integral cohomology group $H^2(\cc_g,\zz)$ of $\cc_g$ contains two well known $\qq$-linearly independent classes $e$, $e_1$, where $e$ denotes the Euler class of the relative tangent bundle $T_{\cc_g/\mm_g}$ of $p \colon \cc_g \to \mm_g$, and where $e_1 = \int_p e^2$ (fiber integral) is the first Miller-Morita-Mumford class in $H^2(\mm_g,\zz)$.

Let $\Sigma$ be a reference surface of genus $g$ and put $H=\H_1(\Sigma,\zz)$, viewed as an $\mathrm{Sp}(2g,\zz)$-representation via the intersection product on $H$, which is unimodular. In the usual way the $\mathrm{Sp}(2g,\zz)$-representation $H$ then gives rise to a local system $\hh_\zz$ of free abelian groups of rank $2g$ on $\mm_g$, and by pullback along $p$ on the universal family $\cc_g$. Write $\hh_\rr=\hh_\zz \otimes \rr$. Then the intersection product on $H$ naturally induces a non-degenerate bilinear alternating pairing $M \colon \hh_\rr \otimes \hh_\rr \to \rr$.

From the local system $\hh_\zz$ one constructs the Jacobian torus bundle $\jj=\hh_\rr/\hh_\zz$, together with its higher versions $\jj_k=\largewedge^{2k+1} \hh_\rr / \largewedge^{2k+1} \hh_\zz$ for $k=1,\ldots,g-1$. The intersection pairing $M$ induces canonical maps $M_k \colon \largewedge^{2k+1} \hh_\rr \otimes \largewedge^{2k+1} \hh_\rr \to \rr$ and hence canonical de Rham cohomology classes in $H^2(\jj_k,\rr)$, which we denote by $\phi_k$. It turns out that the class $\phi_k$ is not integral, but the class $2 \, \phi_k$ is. In a series of papers \cite{familiesI} \cite{familiesII} \cite{molinear} S. Morita has studied the pullbacks of the integral classes $2\, \phi_k$ along certain natural sections of $\jj_k$ over $\cc_g$ (and variants), and in particular expressed them as integral linear combinations of tautological classes.
 
In order to state Morita's results, let $\cc^n_g$ for each $n \in \zz_{\geq 1}$ denote the $n$-fold fiber product of $\cc_g$ over $p \colon \cc_g \to \mm_g$. A first natural section that one considers is the section $\kappa \colon \cc_g \to \jj$ of the universal jacobian bundle $\jj \to \cc_g$ given by the map $ (C,x) \mapsto (J(C),[(2g-2)x-K_C]) $. Here $J(C)$ denotes the jacobian of the compact Riemann surface $C$ and $K_C$ is a canonical divisor on $C$. A second natural section is the well known pointed harmonic volume $I \colon \cc_g \to \jj_1$ studied by B. Harris \cite{harris} and M. Pulte \cite{pu} (cf. Section~\ref{pointed} below). Finally, one has the difference map $ \delta \colon \cc_g^2 \to \jj $ given by $(C,(x,y)) \mapsto (J(C),[y-x])$. It is convenient to view $\delta$ as the universal Abel-Jacobi map sending a point $y$ on a pointed Riemann surface $(C,x)$ to its Abel-Jacobi class $[y-x]$ in $J(C)$.

With the above notation, Morita has obtained the following three identities \cite[Theorem 2.1]{familiesI}, \cite[Theorem 1.7]{familiesII}, \cite[Theorem 5.1]{molinear} in integral cohomology.
\begin{thm} (Morita) Let $p_1, p_2 \colon \cc_g^2 \to \cc_g$ denote the projections onto the first and second coordinate, respectively, and let $\Delta$ be the diagonal on $\cc_g^2$. Then the following relations:
\begin{itemize}
\item[(M1)] $2\,\kappa^* \phi_0 = -2g(2g-2)\,e - e_1$
\item[(M2)] $12\,I^* \phi_1 = -6g \, e + e_1 $
\item[(M3)] $2\,\delta^* \phi_0 = 2\,\Delta - p_1^* e - p_2^* e$
\end{itemize}
hold in $H^2(\cc_g,\zz)$ resp. $H^2(\cc_g^2,\zz)$.
\end{thm}
We note that the formulation in Morita's papers is slightly different: there, the relations (M1)--(M3) are presented as identities for characteristic classes of oriented surface bundles. In particular, one may view these relations as part of the topological study of oriented surfaces, and indeed, the methods employed in \cite{familiesI} \cite{familiesII} \cite{molinear} are distinctively topological. For example, the pointed harmonic volume $I$ realizes the first extended Johnson homomorphism \cite{extension} on the mapping class group of pointed oriented surfaces. 


In the present paper, as opposed to Morita's \cite{familiesI} \cite{familiesII} \cite{molinear}, we would like to take a more analytical point of view. We start out by observing that each of the classes $\phi_k$ is represented by a canonical parallel $2$-form, which we denote by $\omega_k$. The form $\omega_0$ on $\jj$ for example is precisely the symplectic form associated to the intersection product. It follows that the left hand sides of Morita's identities (M1)--(M3) are represented by canonical $2$-forms $2\,\kappa^* \omega_0$, $12\,I^* \omega_1 $ and $2\,\delta^* \omega_0 $.

Combining (M1) and (M2) we find that the $2$-form
\[ e^J = -\frac{1}{2g(2g+1)}(2\,\kappa^*\omega_0+12\,I^*\omega_1)  \]
is a representative of the Euler class $e$. It seems natural to compare the $2$-form $e^J$ with other natural $2$-forms on $\cc_g$ representing $e$. The suggestion of N. Kawazumi \cite{kawhb} \cite{kawpr} is to consider the Arakelov metric on  $T_{\cc_g/\mm_g}$ (cf. Section \ref{Green-metric} below). Introduced by S.Y. Arakelov in \cite{ar} this is a second natural object on $\cc_g$ intimately related to the intersection form on $\hh_\rr$ and hence to the symplectic form $\omega_0$ on the jacobian bundle $\hh_\rr/\hh_\zz$.

Denote by $e^A$ the first Chern form of the Arakelov metric on $T_{\cc_g/\mm_g}$. In the paper \cite{kawpr} Kawazumi determined the difference between the forms $e^A$ and $e^J$. Note that as this difference represents the zero class in cohomology, by general results it can be represented as the $\deldelbar$ of a $C^\infty$ function on $\cc_g$, which is pulled back from $\mm_g$ and uniquely defined up to an additive constant. The question is then: what `is' this function?

Let $C$ be a compact Riemann surface of genus $g \geq 2$. In \cite{kawhb} \cite{kawpr} Kawazumi defines
\begin{equation} \label{defag} a_g(C) = - \sum_{i,j=1}^g \int_C \left( \psi_i \wedge \bar{\psi}_j \Phi(\bar{\psi}_i \wedge \psi_j) + \bar{\psi}_i \wedge \psi_j \Phi(\psi_i \wedge \bar{\psi}_j) \right) \, , 
\end{equation}
where $\Phi \colon D^2(C) \to D^0(C)$ is the Green operator with respect to the Arakelov volume form $\mu$, and where $(\psi_1,\ldots,\psi_g)$ is an orthonormal basis of the space $H^{1,0}$ of holomorphic $1$-forms on $C$ (cf. Section \ref{volume}). The invariant $a_g(C)$ turns out to be a conformal invariant of $C$, which is in fact positive real-valued. One obtains a natural function $a_g \colon \mm_g \to \rr$ on the moduli space of Riemann surfaces, and Kawazumi proved in \cite{kawpr} that this function yields precisely the secondary object relating $e^A$ and $e^J$.
\begin{thm} (Kawazumi, \cite[Theorem 0.1]{kawpr}) \label{kaw} The identity
\[ e^A - e^J = \frac{-2 \sqrt{-1}}{2g(2g+1)} \,\deldelbar \,a_g \]
of $2$-forms holds on $\cc_g$.
\end{thm}
As a variant on the form $e^J$, consider the $2$-form
\begin{equation} \label{e_1^J} e_1^J = \frac{1}{2g+1} (-6\,\kappa^*\omega_0 + 12\, (2g-2)I^*\omega_1 ) \, .\end{equation}
By (M1) and (M2) the form $e_1^J$ represents the tautological class $e_1$. Let $e_1^F = \int_p (e^J)^2$.  Then we have the following variant on Theorem \ref{kaw}, also proved in \cite{kawpr}.
\begin{thm} (Kawazumi, \cite[Theorem 0.2]{kawpr}) \label{secondidentity} The equality
\[ \frac{-2\sqrt{-1}}{2g(2g+1)} \, \deldelbar \, a_g = \frac{1}{(2g-2)^2}(e_1^F - e_1^J) \]
holds on $\cc_g$.
\end{thm}
Unfortunately, the paper \cite{kawpr} has not been published; both results above are stated in \cite{kawhb}. One aim of the present paper is to give another proof of Theorems \ref{kaw} and \ref{secondidentity}. In fact, we will prove slightly more.

Recall that $\Delta$ denotes the diagonal on $\cc_g^2$. We let $h$ denote the first Chern form of the hermitian line bundle $\oo(\Delta)$ on $\cc_g^2$, where the metric on $\oo(\Delta)$ is given by the Arakelov's Green's function $G(x,y)$ (cf. Section \ref{Green-metric}). We have $e^A = h|_\Delta$. Putting
\[ e_1^A = \int_p (e^A)^2   \]
one obtains a natural $2$-form on $\mm_g$ representing the class $e_1$. 

Our main result is then the following. 
\begin{thm} \label{main} Let $p_1, p_2 \colon \cc_g^2 \to \cc_g$ denote the projections onto the first and second coordinate, respectively, and let $\Delta$ be the diagonal on $\cc_g^2$. Let $a_g$ be Kawazumi's invariant (\ref{defag}). Then we have the equalities
\begin{itemize} 
\item[(K1)] $2\,\kappa^* \omega_0 = -2g(2g-2) \, e^A - e_1^A $
\item[(K2)] $12\,I^* \omega_1 = -6g \, e^A + e_1^A -2\sqrt{-1}\,\deldelbar \,a_g $
\item[(K3)] $ 2\,\delta^*\omega_0 = 2\,h - p_1^* e^A - p_2^* e^A $
\end{itemize}
of $2$-forms on $\cc_g$ resp. $\cc^2_g$.
\end{thm}
Note that Theorem \ref{kaw} follows as an immediate consequence of the equalities (K1) and (K2). We will give a short derivation of Theorem \ref{secondidentity} from Theorem \ref{main} at the end of Section \ref{K2}.

The approach in \cite{kawpr} is based on the harmonic Magnus expansion on the universal family of Riemann surfaces. This map defines a flat connection on a vector bundle on the total space of $T_{\cc_g/\mm_g}$ minus the zero-section. The holonomy of this connection gives all the Johnson homomorphisms on the mapping class group $\pi_1(T_{\cc_g/\mm_g} \setminus \textrm{0-section})$. In particular, the first term of the connection form coincides with the first variation $dI \in A^1(\cc_g,\largewedge^3 \hh_\rr)$ of the pointed harmonic volume \cite[Section~8]{magnus}.

Our approach is rather different and more geometric in nature. An important ingredient in our proof of Theorem \ref{main} is the use of the Deligne pairing for families of Riemann surfaces, a refined version of the Gysin map in cohomology  (cf. Section~\ref{sec:deligne} below). Using the Deligne pairing our proofs of especially (K1) and (K3) are remarkably straightforward. For the proof of (K2) we will use, besides the Deligne pairing, the connection discovered by Hain and Pulte \cite{pu} between the pointed harmonic volume and Ceresa's cycle $C_x - C_x^-$ \cite{ce} in the jacobian of a pointed compact Riemann surface $(C,x)$. 

The proof we give of Theorem \ref{main} is independent of Morita's results (M1)--(M3) and in particular gives a new argument for them, by taking classes in cohomology. Another approach is given in the paper \cite{hrgeom} by R. Hain and D. Reed. Because of its fundamental relationship with the cohomology of $\mm_g$, we think that the function $a_g$ merits to be further studied. For example, in his paper \cite{zh} S. Zhang independently introduced the invariant $\varphi_g = 2\pi \, a_g$ and showed the relation of its strict positivity with some important conjectures in arithmetic geometry. In the papers \cite{djsecond}, \cite{djhyp}, \cite{djas}, \cite{djnorm} the present author found a number of further properties of $a_g$, including its asymptotics towards the boundary of $\mm_g$ in the Deligne-Mumford compactification, and closed expressions in the case of a hyperelliptic Riemann surface. We will review these results briefly in our final Section \ref{hainreedfaltings}. 

The starting point of that section is formed by a 1984 paper by G. Faltings \cite{fa} and a 2004 paper by R. Hain and D. Reed \cite{hrar} where one finds a discussion of two other natural invariants in $C^\infty(\mm_g,\rr)$ connected with the geometry of the universal family of (higher) jacobians. In both papers, the determinant of the Hodge bundle with its $L^2$-metric plays a dominant role. In his paper \cite{kawpr}, Kawazumi posed the question of determining the relationship of his invariant $a_g$ with the previously defined invariants $\delta_g$ by Faltings and $\beta_g$ by Hain-Reed. We will show in Section \ref{hainreedfaltings} that one can ``eliminate'' the role of the Hodge bundle, and arrive at a linear dependence relation
\[ \beta_g = \frac{1}{3}\left( 2\pi \, (2g-2) \, a_g + (2g+1) \, \delta_g \right) \]
in $C^\infty(\mm_g,\rr)$. By combining the asymptotics towards the boundary of $\mm_g$ of $a_g$ from our paper \cite{djas} with previously established asymptotics of $\delta_g$ by J. Jorgenson \cite{jo} and R. Wentworth \cite{we}, we thus obtain asymptotics of $\beta_g$. This will furnish an alternative proof of the main result of Hain-Reed's paper \cite{hrar}.  \\

\emph{Conventions.---} For $M$ a complex manifold, a hermitian line bundle on $M$ consists of a holomorphic line bundle $\ll$ on $M$ together with a hermitian norm $\|\cdot \|$ on each fiber of $\ll$, varying in a $C^\infty$ manner over $ M$. More precisely, if $s$ is a local generating section of $\ll$, then the functions $\|s\|^2$ are $C^\infty$. The first Chern form $c_1(\ll)$ of the hermitian line bundle $\ll$ is the real $2$-form on $M$ given locally by the expression $ c_1(\ll) = \frac{\deldelbar}{2\pi \sqrt{-1}} \log \|s\|^2 $, where $s$ is a local generator of $\ll$. Note that the $2$-form $c_1(\ll)$ is independent of the chosen local generator, and represents the first Chern class of $\ll$ in the cohomology group $H^2(M,\rr)$. We use the notation $A^i(M)$ for the space of real valued differential $i$-forms on $M$. If $Z \subset M$ is an analytic subvariety, we denote by $\delta_Z$ the Dirac current associated to $Z$. The current $\delta_Z$ is closed and positive. When $M,N$ are complex manifolds, we say that a surjective map $p \colon M \to N$ is a family if $p$ is a proper holomorphic submersion.

If $V$ is a finite dimensional complex vector space, we shall view $\largewedge^2 V$ as a subspace of $V^{\otimes 2}$ via the embedding $\lambda \wedge \mu \mapsto \lambda \otimes \mu - \mu \otimes \lambda$. For each $m \in \zz$ we have a canonical direct sum decomposition $\largewedge^{2m}(V \oplus V) \isom \bigoplus_{i+j=2m} \largewedge^i V \otimes \largewedge^j V$ of vector spaces. This gives rise to a natural inclusion $\largewedge^2 \largewedge^m V \rightarrowtail (\largewedge^m V)^{\otimes 2} \rightarrowtail \largewedge^{2m}(V \oplus V)$ which we shall denote by $\alpha \mapsto \alpha^\natural$. \\

\emph{Acknowledgments.---} The author thanks Nariya Kawazumi for his many remarks and suggestions.

\section{Compact Riemann surfaces and their jacobians} \label{zero}

Let $C$ be a compact and connected Riemann surface of genus $g \geq 1$ and write $H_\zz=H_1(C,\zz)$. Let $H_\CC=H_\zz \otimes \CC$ and write $H_\CC^*=\mathrm{Hom}_\CC(H_\CC,\CC)$. The Hodge $*$-operator on the real cotangent bundle $T_\rr^*C$ gives rise to a Hodge decomposition
\begin{equation} \label{hodge} H_\CC^* \isom H^{1,0} \oplus H^{0,1} = H^{1,0} \oplus \overline{H^{1,0}} \, .
\end{equation}
The complex vector space $H^{1,0}$ can be identified with the set of holomorphic $1$-forms on $C$, the vector space $H^{0,1}=\overline{H^{1,0}}$ with the set of anti-holomorphic $1$-forms on $C$, and the total space $H_\CC^*$ with the set $\hh$ of harmonic $1$-forms on $C$. Write $H_\rr=H_\zz \otimes \rr$. The intersection pairing $M \colon H_\rr \otimes H_\rr \to \rr$ induces a non-degenerate bilinear alternating pairing on $\hh$, which can be written explicitly as
\[ \alpha \otimes \beta \mapsto \int_C \alpha \wedge \beta  \]
for harmonic $1$-forms $\alpha,\beta$ on $C$.

 The jacobian of $C$ is defined to be the real torus $J=H_\rr/H_\zz$. Let $k \in \zz$ be an integer. We write $M_k$ for the canonically induced bilinear alternating map $\largewedge^{2k+1} H_\rr \otimes \largewedge^{2k+1} H_\rr \to \rr$; explicitly $M_k$ is given by
\[ M_k(a_1 \wedge \ldots \wedge a_{2k+1} \otimes b_1 \wedge \ldots \wedge b_{2k+1}) = \det(M(a_i \otimes b_j))  \]
for $a_1,\ldots,a_{2k+1},b_1,\ldots,b_{2k+1} \in H_\rr$.
Note that we have $M=M_0$. Let $J_k = \largewedge^{2k+1} H_\rr / \largewedge^{2k+1} H_\zz$ for $k \in \zz$ be the ``higher'' intermediate jacobians. Then for each $k \in \zz$ we define $\omega_k \in A^2(J_k)$ to be the canonical translation-invariant $2$-form corresponding to $M_k$.

Let $(\ell_1,\ldots,\ell_{2g})$ be a symplectic basis of $H_\rr$, that is
\[ M(\ell_i \otimes \ell_{g+j}) = \delta_{ij} \, , \quad M(\ell_i \otimes \ell_j) = M(\ell_{g+i} \otimes \ell_{g+j}) =0 \, , \quad 1 \leq i,j \leq g \, . \]
Let $(\lambda_1,\ldots,\lambda_{2g})$ be the dual basis of $H_\rr^*=\mathrm{Hom}_\rr(H_\rr,\rr)$. It is useful to have an expression for $\omega_k$ in terms of this basis.
\begin{prop}  \label{pre-explicit}
We have the identity
\[ 2\, \omega_k = \sum_{1 \leq i_1 < \ldots < i_{2k+1} \leq 2g} \lambda_{i_1} \wedge\ldots \wedge \lambda_{i_{2k+1}} \wedge \lambda_{g+i_1} \wedge \ldots \wedge \lambda_{g+i_{2k+1}} \]
in $A^2(J_k)$. Here we put $\lambda_{2g+j}=-\lambda_j$ for $j=1,\ldots,g$.
\end{prop}
\begin{proof} This is a straightforward verification.
\end{proof}
Next, let $\largewedge^2 \largewedge^{2k+1} H_\rr^* \to \largewedge^{4k+2} ( H_\rr^* \oplus H_\rr^*)$ be the natural map given by sending $\alpha \mapsto \alpha^
\natural$ as in Section \ref{intro}. Note that $H_\rr \oplus H_\rr$  uniformizes the selfproduct $J \times J$. We can thus naturally interpret $\omega_k^\natural$ as an element of $A^{4k+2}(J \times J)$.
Written out explicitly, by Proposition \ref{pre-explicit} we have
\[ 2 \, \omega_k^\natural = \sum_{1 \leq i_1 < \ldots < i_{2k+1} \leq 2g} \lambda_{i_1}(u) \wedge\ldots \wedge \lambda_{i_{2k+1}}(u) \wedge \lambda_{g+i_1}(v) \wedge \ldots \wedge \lambda_{g+i_{2k+1}}(v) \]
for $(u,v) \in J \times J$. In particular, the restriction of $\omega_k^\natural$ to the diagonal gives back $\omega_k$.

Note that we have a Hodge filtration
\[ H_\CC = F^{-1}H_\CC \supset F^0H_\CC \supset F^1H_\CC=0   \]
dual to the one in (\ref{hodge}).
The inclusion $H_\rr \to H_\CC$ induces a natural isomorphism
\begin{equation} \label{natural} H_\rr/H_\zz \isom H_\CC / (F^0 H_\CC + H_\zz) 
\end{equation}
of real tori. The projection onto $H^{-1,0}$ induces an isomorphism $H_\CC/F^0 H_\CC \isom H^{-1,0} = (H^{1,0})^*$. Hence by the natural isomorphism (\ref{natural}) the jacobian $J$ of $C$ can be naturally viewed as a complex torus, whose tangent space at the origin is naturally identified with $H^{-1,0}=(H^{1,0})^*$. More precisely we have $J = (H^{1,0})^*/H_\zz$ as complex tori, where $H_\zz$ is embedded in $(H^{1,0})^*$ by associating to the class of a $1$-cycle $\gamma$ on $C$ the linear functional given by $\int_\gamma \cdot$ on $H^{1,0}$.

A similar description can be carried through for the higher jacobians $J_k$. Using the Hodge filtration on $\largewedge^{2k+1} H_\CC \simeq H_{2k+1}(J,\CC)$ one has a natural identification of $J_k$ with the complex torus
$ ( \bigoplus_{p+q=2k+1  , \, p \geq q} H^{p,q} )^*/H_{2k+1}(J,\zz)$,
where $H^{p,q}$ is the space of harmonic $(p,q)$-forms on $J$. Here the homology group $H_{2k+1}(J,\zz)$ is embedded into $ ( \bigoplus_{p+q=2k+1  , \, p \geq q} H^{p,q} )^* $ by associating to the class of a $(2k+1)$-cycle $Z$ in $J$ the functional given by $\int_Z \cdot$ on $\bigoplus_{p+q=2k+1  , \, p \geq q} H^{p,q}$.

To finish this section we recall a construction due to P. Griffiths \cite{gr} that, restricted to our situation, associates to a homologically trivial cycle $\Gamma$ of dimension $2k$ in $J$ a canonical class in the $k$-th higher jacobian $J_k$. First we write the cycle $\Gamma$ as the boundary $\partial \tilde{\Gamma}$ of a $(2k+1)$-cycle $\tilde{\Gamma}$ in $J$. Then we associate to $\Gamma$ the class $\gamma$ in $J_k$  given by the functional $\int_{\tilde{\Gamma}} \cdot$ on $\bigoplus_{p+q=2k+1  , \, p \geq q} H^{p,q}$. Note that $\gamma$ is well-defined: replacing $\tilde{\Gamma}$ by another $(2k+1)$-cycle $\Gamma'$ such that $\partial \Gamma'= \Gamma$, the functional changes by $\int_{\tilde{\Gamma} - \Gamma'} \cdot$, which has zero class in $J_k$, as $\tilde{\Gamma} - \Gamma'$ yields an element of $H_{2k+1}(J,\zz)$.

Griffiths's construction generalizes the classical Abel-Jacobi map (where $k=0$) which assigns to a degree-zero divisor $D$ on a compact Riemann surface $C$ the class of the functional $\int_{\tilde{D}} \cdot$ in $J$, where $\tilde{D}$ is a $1$-cycle in $C$ with $\partial \tilde{D} = D$. Of particular interest for us will be the class $\gamma$ in $J_1$ associated by Griffiths's construction to the so-called Ceresa cycle \cite{ce} on the pointed Riemann surface $(C,x)$. This cycle is given as follows: let $C_x$ be the copy of $C$ embedded in $J$ given by all Abel-Jacobi images of  degree-zero divisors $y-x$ on $C$, where $y$ runs through $C$, and let $C_x^-$ be the image of $C_x$ under the inversion $[-1]$ on $J$. Then the Ceresa cycle of $(C,x)$ is defined to be the cycle $C_x-C_x^-$ in $J$. Since inversion acts by $+1$ on $H_2(J,\zz)$, we see that the Ceresa cycle is homologically trivial in $J$.  

\section{Torus bundles and sections} \label{prelims}

Let $B$ be a complex manifold, and let $q \colon \tt \to B$ be a real torus bundle. Let $\hh_\zz$ be the associated local system of first homology groups. That is, the fiber of $\hh_\zz$ at $b \in B$ is the free abelian group $H_1(\tt_b,\zz)$. Write $\hh_\rr = \hh_\zz \otimes \rr$, and assume that a zero section of $q$ has been given. Then there is a natural identification $\tt = \hh_\rr/\hh_\zz$.

Let $M \colon \hh_\rr \otimes \hh_\rr \to \rr$ be a bilinear alternating pairing. As for each $b \in B$ the vector space $\hh_{\rr,b}$ is identified with the relative tangent space to $\tt_b$ at the origin, the form $M$ naturally gives rise to a parallel $2$-form $\omega$ on $\tt$. It is clear that the pullback of $\omega$ along the zero section of $\tt \to B$ is trivial, and that the restriction of $\omega$ to any of the fibers of $\tt \to B$ is translation-invariant.

Let $s \in A^0(B,\tt)$ be a $C^\infty$ section of $\tt \to B$. Then associated to $s$ we have its first variation $\d s \in A^1(B,\hh_\rr)$.  It can be easily seen that $ds$ is closed and that its cohomology class gives rise to an element of $H^1(B,\hh_\zz)$ \cite[Section~4.1]{hain_normal}. Further, the pairing $M$ gives rise to a contraction map $M \colon A^2(B,(\hh_\rr)^{\otimes 2}) \to A^2(B)$.  A straightforward computation yields the identity
\begin{equation} \label{rewrite} s^* \omega = M(\d s)^{\otimes 2}
\end{equation}
of $2$-forms on $B$.

Now let $p \colon \cc \to B$ be a family of compact Riemann surfaces of genus $g$, with associated local system $\hh_\zz$ of first homology groups on $B$. Then $\hh_\zz$ is a variation of Hodge structure of weight $-1$ on $B$ with Hodge filtration
\[ \hh_\CC = \ff^{-1}\hh_\CC \supset \ff^0 \hh_\CC \supset \ff^1 \hh_\CC=0 \, .  \]
By our remarks in Section \ref{zero} the jacobian bundle $q \colon \jj = \hh_\rr/\hh_\zz \to B$ can be naturally viewed as a bundle of complex tori. 

Let $M \colon \hh_\rr \otimes \hh_\rr \to \rr$ be the nondegenerate bilinear alternating pairing derived from the intersection form, and for each $k \in \zz$ write $M_k$ for the canonically induced bilinear alternating pairing $\largewedge^{2k+1} \hh_\rr \otimes \largewedge^{2k+1} \hh_\rr \to \rr$. Let $\jj_k = \largewedge^{2k+1} \hh_\rr / \largewedge^{2k+1} \hh_\zz$ for $k \in \zz$ be the higher jacobian  bundles over $B$. Then we define $\omega_k \in A^2(\jj_k)$ for $k \in \zz$ to be the parallel $2$-form corresponding to $M_k$. 

Let $s \in A^0(B,\jj_k)$ be a $C^\infty$ section. Then by equation (\ref{rewrite}) we can write
\begin{equation} \label{section*} s^* \omega_k = M_k (\d s)^{\otimes 2} 
\end{equation}
in $A^2(B)$. In particular we can write our three main $2$-forms of interest as:
\[ 2\,\kappa^*\omega_0=2\,M_0(\d \kappa)^{\otimes 2} \, , \quad 12 \,I^*\omega_1=12\,M_1(\d I)^{\otimes 2} \, , \quad  2\,\delta^*\omega_0=2\,M_0(\d \delta)^{\otimes 2} \, . \]
Here, by a slight abuse of notation we view $\delta$ as a section of the universal family of jacobians over $\cc_g^2$. Working locally over $B$, let $(\ell_1,\ldots,\ell_{2g})$ be a symplectic frame of $\hh_\rr$, that is
\[ M(\ell_i \otimes \ell_{g+j}) = \delta_{ij} \, , \quad M(\ell_i \otimes \ell_j) = M(\ell_{g+i} \otimes \ell_{g+j}) =0 \, , \quad 1 \leq i,j \leq g \, . \]
Let $(\lambda_1,\ldots,\lambda_{2g})$ be the dual frame of $\hh_\rr^*$. Then Proposition \ref{pre-explicit} immediately translates into the following.
\begin{prop} \label{explicit}
Locally over $B$, the equality
\[ 2\,\omega_k = \sum_{1 \leq i_1 < \ldots < i_{2k+1} \leq 2g} \lambda_{i_1} \wedge\ldots \wedge \lambda_{i_{2k+1}} \wedge \lambda_{g+i_1} \wedge \ldots \wedge \lambda_{g+i_{2k+1}} \]
holds in $A^2(\jj_k)$. Here we put $\lambda_{2g+j}=-\lambda_j$ for $j=1,\ldots,g$.
\end{prop}
Note that in particular $\omega_k$ can be interpreted as an element of $A^0(B,\largewedge^2 \largewedge^{2k+1} \hh_\rr^*)$. Let $\largewedge^2 \largewedge^{2k+1} \hh_\rr^* \to \largewedge^{4k+2} ( \hh_\rr^* \oplus \hh_\rr^*)$ be the natural map induced by sending $\alpha \mapsto \alpha^
\natural$ as in Section \ref{intro}. As the local system $\hh_\rr \oplus \hh_\rr$  uniformizes the selfproduct of jacobian bundles $Q \colon \jj \times_B \jj \to B$ we can naturally interpret each $\omega_k^\natural$ as an element of $A^{4k+2}(\jj \times_B \jj)$. Written out explicitly by Proposition \ref{explicit} we have
\[ 2\,\omega_k^\natural = \sum_{1 \leq i_1 < \ldots < i_{2k+1} \leq 2g} \lambda_{i_1}(u) \wedge\ldots \wedge \lambda_{i_{2k+1}}(u) \wedge \lambda_{g+i_1}(v) \wedge \ldots \wedge \lambda_{g+i_{2k+1}}(v) \]
for $(u,v) \in \jj \times_B \jj$, locally over $B$. In particular, the restriction of $\omega_k^\natural$ to the diagonal gives the $2$-form $\omega_k$.

To finish this section, we discuss a useful differential property of the Griffiths Abel-Jacobi construction in the family of complex tori $q \colon \jj \to B$. Let $k \in \zz$ and let $\Gamma$ be a topological cycle in $\jj$ whose restriction to each fiber of $q \colon \jj \to B$ is homologically trivial of dimension $2k$. Note that for each $b \in B$ there exists an open neighborhood $U$ of $b$ in $B$ and a topological cycle $\tilde{\Gamma}$ in $\jj|_{\overline{U}}$ such that $\Gamma|_{\overline{U}} = \partial \tilde{\Gamma}$. Hence, applying fiberwise the Griffiths Abel-Jacobi construction from Section \ref{zero} we obtain a natural section $\gamma \in A^0(B,\jj_k)$ associated to $\Gamma$. The following proposition describes the canonical element of $A^2(B)$ induced by $\gamma$ and $M$ as a fiber integral.
\begin{prop} \label{fiberint} Let $\Gamma$ be  a topological cycle in $\jj$ whose restriction to each fiber of $q \colon \jj \to B$ is homologically trivial of dimension $2k$.
Let $\gamma \in A^0(B,\jj_k)$ be the section obtained by fiberwise applying Griffiths's Abel-Jacobi construction to $\Gamma$. Let $Q \colon \jj \times_B \jj \to B$ be the projection map. Then the equalities
\[ \gamma^*(\omega_k)=M_k(\d \gamma)^{\otimes 2} = \int_Q \omega_k^
\natural \wedge \delta_{\Gamma \times_B \Gamma} \]
hold in $A^2(B)$.
\end{prop}
\begin{proof} The first equality is just equation (\ref{section*}). To see the second equality, let $\alpha$ be any element of $A^0(B,\largewedge^{2k+1} \hh_\rr^*)$. Note that we can interpret $\alpha$ as a parallel $(2k+1)$-form on $\jj$. By contraction we find an element $\alpha(d\gamma) \in A^1(B)$.  Working locally over $B$, let $\tilde{\Gamma}$ be a cycle on $\jj$ such that $\partial \tilde{\Gamma} = \Gamma$. An application of Stokes's theorem then yields that
\[ \alpha(\d \gamma) = \d \int_q \alpha \wedge \delta_{\tilde{\Gamma}} = -\int_q \alpha \wedge \delta_{\Gamma}  \]
in $A^1(B)$, as $\alpha$ is closed. It follows that for $\alpha, \alpha' 
\in A^0(B,\largewedge^{2k+1} \hh_\rr^*)$ we have
\[ (\alpha \wedge \alpha')(d\gamma)^{\otimes 2} 
= \left(- \int_q \alpha \wedge \delta_{\Gamma} \right) \wedge \left(-\int_q \alpha' \wedge \delta_{\Gamma} \right) = \int_Q \alpha(u) \wedge \alpha'(v) \wedge \delta_{\Gamma \times_B \Gamma} \]
in $A^2(B)$. The proposition follows.
\end{proof}

\section{Arakelov volume form and the invariant $a_g$} \label{volume}

Let $g \geq 1$ be an integer. Let $C$ be a compact and connected Riemann surface of genus $g$. The purpose of this section is to define the Arakelov volume form $\mu$ and the conformal invariant $a_g(C)$ of $C$ introduced by Kawazumi. Our exposition here is based on the discussion in \cite[Section~1]{kawpr}.

As before let $H_\CC=H_1(C,\zz)$ and let $\hh \simeq \mathrm{Hom}_\CC(H_\CC,\CC)$ be the complex vector space of harmonic $1$-forms on $C$ with its Hodge decomposition
\[ \hh \isom H^{1,0} \oplus H^{0,1} = H^{1,0} \oplus \overline{H^{1,0}} \, . \]
Let $(\psi_1,\ldots,\psi_g)$ be an orthonormal basis of $H^{1,0}$, that is,
\[ \frac{\sqrt{-1}}{2} \int_C \psi_i \wedge \bar{\psi}_j = \delta_{ij} \, , \quad 1 \leq i,j \leq g \, . \]
Then $(\psi_1,\ldots,\psi_g,\frac{\sqrt{-1}}{2}\bar{\psi}_1,\ldots,\frac{\sqrt{-1}}{2}\bar{\psi}_g)$ is a symplectic basis of $\hh$ and it follows from Proposition \ref{explicit} that the form induced by $M$ in $ \largewedge^2 \hh$ can be written explicitly as
\[ M=\frac{\sqrt{-1}}{2} \sum_{i=1}^g \psi_i \wedge \bar{\psi_i} \, . \]
Let $H^*_\CC \to A^1(C)$ be the map sending a linear form to the harmonic $1$-form that corresponds to it. It is convenient to see this map as an $H_\CC$-valued real harmonic $1$-form $\omega_{(1)} \in A^1(C) \otimes H_\CC$. Letting $(Y_1,\ldots,Y_g,Y_{g+1},\ldots,Y_{2g})$ be the dual basis of $(\psi_1,\ldots,\psi_g,\frac{\sqrt{-1}}{2}\bar{\psi}_1,\ldots,\frac{\sqrt{-1}}{2}\bar{\psi}_g)$ in $H_\CC$ we can then write explicitly
\[ \omega_{(1)} = \sum_{i=1}^g \psi_i Y_i + \overline{ \psi_i Y_i } \, . \]
Regarding the $\psi_i$ as holomorphic $1$-forms on $C$, this naturally leads to considering the $2$-form
\[ \mu = \frac{1}{2g} M(\omega_{(1)} \wedge \omega_{(1)}) = \frac{\sqrt{-1}}{2g} \sum_{i=1}^g \psi_i \wedge \bar{\psi}_i \]
on $C$, where we note that $\int_C \mu =1$. By the Riemann-Roch theorem the canonical linear system has no base points, and hence the $2$-form $\mu$ is a volume form on $C$. We call $\mu$ the Arakelov volume form.

Let $D^q(C)$ for $q \in \zz$ be the space of complex valued $q$-currents on $C$. In order to introduce the invariant $a_g(C)$ of $C$, we recall the Green operator $\Phi \colon D^2(C) \to D^0(C)$ with respect to the Arakelov volume form $\mu$. It is defined uniquely by the following two properties:
\begin{equation} \label{Phi} d * d \, \Phi(\Omega) = \Omega - \left( \int_C \Omega \right) \mu \, , \quad \int_C \mu \, \Phi(\Omega)  = 0
\end{equation}
for all $\Omega \in D^2(C)$. If $\Omega$ is a smooth form, then $\Phi(\Omega)$ is a smooth function. Moreover, we have the symmetry relation
\begin{equation} \label{symmetry} \int_C \Phi(\Omega') \, \Omega = \int_C \Omega'\,\Phi(\Omega)
\end{equation}
for all $\Omega, \Omega'\in D^2(C)$.
\begin{defini} (cf. \cite[Section~8]{kawhb}, \cite[Section~1]{kawpr}) Let $(\psi_1,\ldots,\psi_g)$ be an orthonormal basis of the space $H^{1,0}$ of holomorphic $1$-forms on $C$. The invariant $a_g(C)$ of $C$ is defined to be the element
\[ \begin{split}
a_g(C) & = (M \otimes M) \int_C \omega_{(1)} \, \Phi(\omega_{(1)} \wedge \omega_{(1)}) \, \omega_{(1)} \\
 & = - \sum_{i,j=1}^g \int_C \left( \psi_i \wedge \bar{\psi}_j \Phi(\bar{\psi}_i \wedge \psi_j) + \bar{\psi}_i \wedge \psi_j \Phi(\psi_i \wedge \bar{\psi}_j) \right) \end{split} \]
of $\rr$.  
\end{defini}
We have the following positivity result for $a_g(C)$.
\begin{prop} (cf. \cite[Corollary 1.2]{kawpr}) We have $a_g(C) \geq 0$, and $a_g(C)=0$ if and only if $g=1$.
\end{prop}
\begin{proof} Let $\Omega \in D^2(C)$. Then we have
\[ - \int_C \Omega \, \Phi(\overline{\Omega}) = \sqrt{-1} \int_C \partial \Phi(\Omega) \wedge \overline{\partial} \Phi(\overline{\Omega}) + \sqrt{-1} \int_C \partial \Phi(\overline{\Omega}) \wedge \overline{\partial} \Phi(\Omega) \]
and hence $ - \int_C \Omega \, \Phi(\overline{\Omega}) \geq 0$.
Applying this to $\Omega = \psi_i \wedge \bar{\psi}_j$ for $1 \leq i,j \leq g$ we obtain the first part of the proposition.  
Assume that $\int_C \Omega \, \Phi(\overline{\Omega})=0$. Then both $\partial \Phi(\Omega)$ and $\overline{\partial} \Phi(\Omega)$ vanish and hence $\Phi(\Omega)$ is a constant. From (\ref{Phi}) we infer that $\Omega$ is a scalar multiple of $\mu$. We find that if $a_g(C)=0$, then each $\psi_i \wedge \bar{\psi}_j$ is a scalar multiple of $\mu$. It follows that $g=1$. Vice versa we have that $a_g$ vanishes if $g=1$.
\end{proof}
We mention that the strict positivity for $g \geq 2$ for $\varphi_g = 2\pi \, a_g$ has also been obtained by S. Zhang in \cite[Proposition 2.5.3]{zh}. He uses this result in \cite{zh} to give a surprising application in arithmetic geometry: for a smooth, projective and geometrically connected curve of genus $\geq 2$ defined over a number field, the truth of an arithmetic version of a standard conjecture of Hodge index type due to Grothendieck implies the Bogomolov conjecture. 

\section{Arakelov Green's function and Arakelov metric} \label{Green-metric}

In this section we introduce the Arakelov-Green's function on a compact Riemann surface $C$ of genus $g \geq 1$. References for this section are Arakelov's original paper \cite[Sections 3-4]{ar} and \cite[Section~2]{we}.

For $x \in C$ denote by $\delta_x \in D^2(C)$ the Dirac current supported at $x$.
Let $\Phi$ denote the Green operator with respect to the Arakelov volume form $\mu$. The Arakelov Green's function $G(x,\cdot)$ is then defined to be the function on $C$ given by
\[ G(x,\cdot) = \exp(2\pi \Phi(\delta_x)) \, . \]
The square of $G(x,\cdot)$ is a $C^\infty$-function on $C$, vanishing only at $x$. From (\ref{Phi}) we obtain that the Arakelov Green's function satisfies the conditions
\begin{equation} \label{conditions} \frac{-\deldelbar}{\pi \sqrt{-1}} \log G(x,\cdot) = \delta_x - \mu \, , \quad \int_C \log G(x,\cdot) \, \mu = 0 \, , 
\end{equation}
and in fact these conditions uniquely determine $\log G(x,\cdot)$ as an element of $D^0(C)$.
Further, applying (\ref{symmetry}) to $\Omega' = \delta_x$ we obtain
\begin{equation} \label{PhiG} \Phi(\Omega)(x) = \frac{1}{2\pi} \int_C \log G(x, \cdot) \, \Omega
\end{equation}
for all $\Omega \in D^2(C)$.

Let $\Delta$ be the diagonal on $C\times C$. We use $G$ to put a hermitian metric on $\oo(\Delta)$, as follows. Let $1$ denote the canonical meromorphic section of $\oo(\Delta)$. Then we demand that $\|1\|^2(x,y) = G(x,y)^2$. As the normal bundle to the diagonal equals the tangent bundle $TC$ of $C$, we obtain a natural metric on $TC$. We call this the Arakelov metric. We put
\[ h = c_1(\oo(\Delta)) \in A^2(C \times C) \, .  \]
It follows that the identity of currents
\[ \frac{-\deldelbar}{\pi \sqrt{-1}} \log G = \delta_\Delta - h \]
holds on $C \times C$.
In \cite[Proposition 3.1]{ar} one finds the explicit formula 
\begin{equation} \label{explicith}
 h(x,y)= \mu(x) + \mu(y) - \sqrt{-1} \sum_{i=1}^g (\psi_i(x) \bar{\psi}_i(y)+\psi_i(y) \bar{\psi}_i(x)) 
\end{equation}
for the $2$-form $h$. In particular we find that
\begin{equation} \label{c1tangent} c_1(TC) = h|_\Delta = (2-2g)\,\mu \, .
\end{equation}
The following proposition expresses $a_g(C)$ explicitly in terms of $G$ and $h$.
\begin{prop} \label{formula_ag} The formula
\[ a_g(C) = \frac{1}{2\pi} \int_{C \times C} \log G \, h^2 \]
holds.
\end{prop}
\begin{proof} By (\ref{PhiG}) we can write
\[ \begin{split}
a_g(C) & =  - \sum_{i,j=1}^g \int_C \left( \psi_i \wedge \bar{\psi}_j \Phi(\bar{\psi}_i \wedge \psi_j) + \bar{\psi}_i \wedge \psi_j \Phi(\psi_i \wedge \bar{\psi}_j) \right) \\
& = - \frac{1}{2\pi} \sum_{i,j=1}^g \int_{C \times C} \log G \left(\psi_i(x)  \bar{\psi}_j(x) \bar{\psi}_i(y) \psi_j(y) + \bar{\psi}_i(x) \psi_j(x) \psi_i(y)  \bar{\psi}_j(y) \right)   \, .  \end{split} \]
From (\ref{explicith}), a small calculation yields
\[
 h^2 = 2 \, \mu(x) \, \mu(y) - \sum_{i,j=1}^g \left(\psi_i(x)  \bar{\psi}_j(x) \bar{\psi}_i(y) \psi_j(y) + \bar{\psi}_i(x) \psi_j(x) \psi_i(y)  \bar{\psi}_j(y) \right) \, . \]
The proposition follows by noting that $\int_C \log G(x,\cdot)\,\mu=0$.
\end{proof}
\begin{lem} \label{selfinters} Let $p_1,p_2 \colon C \times C \to C$ be the projection on the first and second coordinate, respectively. The form $h$ restricts to the Arakelov form on any fiber of $p_1$ or $p_2$. Further, we have the equalities
\[ h \,p_1^*\mu = h \,p_2^* \mu = p_1^*\mu \, p_2^* \mu \]
of $4$-forms on $C \times C$. Finally, the identity
\[ \int_{C \times C} (h - p_1^* \mu - p_2^* \mu)^2 = -2g  \]
holds.
\end{lem}
\begin{proof} All assertions can be verified directly from formula (\ref{explicith}) for $h$.
\end{proof}
We globalize the above discussion as follows. Let $p \colon \cc_g \to \mm_g$ be the universal family of Riemann surfaces of genus $g$. First, the Arakelov-Green's function defines a function $G \colon \cc_g^2 \to \rr$, and a natural hermitian metric on the diagonal line bundle $\oo(\Delta)$ on $\cc_g^2$. As the relative tangent bundle $T_{\cc_g/\mm_g}$ is precisely the normal bundle of the diagonal in $\cc_g^2$, this induces a natural Arakelov metric on $T_{\cc_g/\mm_g}$. Denote by $h \in A^2(\cc^2_g)$ the first Chern form of $\oo(\Delta)$. In particular we have an equality
\begin{equation} \label{2ndvarG} \frac{-\deldelbar}{\pi \sqrt{-1}} \log G = \delta_\Delta - h 
\end{equation}
of $2$-currents on $\cc_g^2$.

We define
\[ e^A = h|_\Delta \, , \]
i.e. the restriction of $h$ to the diagonal. Then $e^A$ is the first Chern form of $T_{\cc_g/\mm_g}$ with its Arakelov metric. If $C$ is a fiber of $\cc_g \to \mm_g$ with Arakelov volume form $\mu$ we obtain
\[ e^A|_C = (2-2g)\mu  \]
from (\ref{c1tangent}). Write $F$ for the canonical map $ \cc_g^2 \to \mm_g$. The following proposition computes the second variation of $a_g$ on $\mm_g$ in terms of fiber integrals.
\begin{prop} \label{a_gfiberint} The equality
\[ -2\sqrt{-1}\, \deldelbar \, a_g = \int_F h^3 - e_1^A  \]
holds in $A^2(\mm_g)$.
\end{prop}
\begin{proof}
By Proposition \ref{formula_ag} and identity (\ref{2ndvarG}) we have
\[ \begin{split}
-2\sqrt{-1}\, \deldelbar \, a_g &= \frac{\deldelbar}{\pi\sqrt{-1}} \int_{F} \log G \, h^2  =  \int_{F} \frac{\deldelbar}{\pi\sqrt{-1}} \log G \, h^2 \\
& = \int_F (h - \delta_\Delta)h^2 = \int_F h^3 - \int_p (h|_\Delta)^2 \\
& = \int_F h^3 - \int_p (e^A)^2 = \int_F h^3 - e_1^A \, .
\end{split}
\]
This proves the proposition.
\end{proof}

\section{Deligne pairing} \label{sec:deligne}

The aim of this section is to introduce the Deligne pairing for hermitian line bundles on a family of compact Riemann surfaces. References for this section are \cite[Section XIII.5]{acg} and \cite[Section~6]{de}. 

Let $B$ be a complex manifold and let $p \colon \cc \to B$ be a family of compact Riemann surfaces. Let $\ll, \mm$ be two holomorphic line bundles on $\cc$. To these data one can canonically associate a holomorphic line bundle denoted $\langle \mathcal{L},\mathcal{M} \rangle$ on $B$, in the following way: local generators of $\langle \mathcal{L},\mathcal{M} \rangle$ are formal symbols $\langle l,m \rangle$, where $l,m$ are local generating sections of $\mathcal{L}, \mathcal{M}$, subject to the following relations:
\[ \langle l, fm \rangle = f[\operatorname{div}{l}] \cdot \langle l,m \rangle \, , \quad
\langle fl, m \rangle = f[\operatorname{div}{m}] \cdot \langle l, m \rangle \, ,  \]
for all germs of holomorphic functions $f$. Here, if $D=\sum_x n_x x$ is a divisor on a Riemann surface $C$, and $g$ is a function on $C$ whose singularities are disjoint from $D$, we write $g[D]$ as a shorthand for $\sum_x n_x g(x)$. An application of Weil's reciprocity law shows that the above relations are well-defined. We call $\pair{\mathcal{L},\mathcal{M}}$ the Deligne pairing associated to $\mathcal{L},\mathcal{M}$.

The Deligne pairing is biadditive and symmetric, in the sense that for holomorphic line bundles $\mathcal{L}_1$, $\mathcal{L}_2$, $\mathcal{M}_1$, $\mathcal{M}_2$, $\mathcal{L}$, $\mathcal{M}$ on $\cc$ we have canonical isomorphisms \[ \langle \mathcal{L}_1 \otimes \mathcal{L}_2, \mathcal{M} \rangle \isom \langle \mathcal{L}_1 , \mathcal{M} \rangle \otimes \langle
\mathcal{L}_2 , \mathcal{M} \rangle \, , \, \langle \mathcal{L}, \mathcal{M}_1 \otimes \mathcal{M}_2 \rangle \isom \langle \mathcal{L}, \mathcal{M}_1 \rangle \otimes \langle \mathcal{L} , \mathcal{M}_2 \rangle \, , \]
and $\langle \mathcal{L},\mathcal{M} \rangle \isom \langle \mathcal{M} , \mathcal{L}\rangle$. An isomorphism $\mathcal{L}_1 \isom \mathcal{L}_2$ of holomorphic line bundles induces a natural isomorphism $\langle \mathcal{L}_1, \mathcal{M} \rangle \isom \langle \mathcal{L}_2, \mathcal{M} \rangle$.
Finally, if $x \colon B \to \cc$ is a holomorphic section of $p$, then there is a canonical isomorphism $\langle \oo_\cc(x), \ll \rangle \isom x^* \ll$ of line bundles on $B$.

If we suppose that both $\mathcal{L}, \mathcal{M}$ are equipped with a $C^\infty$-hermitian metric, then according to \cite[Section~6.3]{de} we can put a natural associated hermitian metric $\|\cdot\|$ on the Deligne pairing $\langle \mathcal{L} , \mathcal{M} \rangle$, by requiring the identity
\begin{equation} \label{metric} \log \| \langle l,m \rangle \| =  \log \| m \| [\operatorname{div}{l}] + \int_p \log \|l\| \, c_1(\mm)
\end{equation}
to hold for all local generating sections $l$, $m$ of $\ll$, $\mm$ (with disjoint support). It is shown in \cite[Section~6]{de} that $\|\cdot \|$ provides $\langle \mathcal{L},\mathcal{M} \rangle$ with a $C^\infty$-hermitian metric, and that each of the canonical isomorphisms mentioned above is an isometry. Further, let $x \colon B \to \cc$ be a holomorphic section and endow the line bundle $\oo_\cc(x)$ with the Arakelov metric given by putting $\|1\|(y)=G(x,y)$, where $G$ is the Arakelov-Green's function on $\cc \times_B \cc$, and let $\ll$ be a hermitian line bundle whose first Chern form restricts to a multiple of the Arakelov form $\mu$ in each fiber of $p \colon \cc \to B$. One then checks easily from (\ref{conditions}) and (\ref{metric}) that the canonical isomorphism $\langle \oo_\cc(x), \ll \rangle \isom x^* \ll$ is an isometry.

From (\ref{metric}) one derives the following expression for the first Chern form of the Deligne pairing (cf. \cite[Proposition 6.6]{de}).
\begin{prop} \label{c1deligne}
Let $p \colon \cc \to B$ be a family of compact Riemann surfaces, and $\ll$ and $\mm$ two hermitian line bundles on $\cc$. Let $\langle \ll, \mm \rangle$ be the Deligne pairing of $\ll, \mm$ on $B$, equipped with its hermitian metric determined by (\ref{metric}). Then the equality of differential forms
\[ c_1( \langle \ll, \mm \rangle) = \int_p c_1(\ll) \wedge c_1(\mm)  \]
holds in $A^2(B)$.
\end{prop}
In particular, denoting by $a, b \in H^2(\cc,\zz)$ the cohomology classes of $\ll, \mm$, the first Chern form of the Deligne pairing $\langle \ll,\mm \rangle$ represents the Gysin image $p_{\,!}(a b) \in H^2(B,\zz)$.

As a first application of Proposition \ref{c1deligne} let $p_1 \colon \cc_g^2 \to \cc_g$ be the projection on the first coordinate and let $x \colon \cc_g \to \cc_g^2$ be the tautological section of $p_1$. Note that $x$ can be identified with the diagonal embedding $\Delta \colon \cc_g  \to \cc_g^2$. On $\cc_g$ we then have canonical isometries
\begin{equation} \label{Deligne} \langle \oo(x), \oo(x) \rangle = \langle \oo(\Delta), \oo(\Delta) \rangle \isom \Delta^* \oo(\Delta) \isom T_{\cc_g/\mm_g} \, ,
\end{equation}
where $T_{\cc_g/\mm_g}$ is equipped with the Arakelov metric. Recall that we denote by $h \in A^2(\cc^2_g)$ the first Chern form of $\oo(\Delta)$, and by $e^A\in A^2(\cc_g)$ the first Chern form of $T_{\cc_g/\mm_g}$. From the above isometries we obtain by Proposition \ref{c1deligne} the following identity.
\begin{prop} \label{eA_as_fiberintegral}
The equality
\[  \int_{p_1} h^2 = e^A  \]
of $2$-forms holds on $\cc_g$.
\end{prop}
As a second application, we relate the first Chern form of the Deligne pairing of two relatively flat line bundles on a family $p \colon \cc \to B$ of compact Riemann surfaces to the form $\omega_0^\natural$ introduced in Section \ref{prelims}. As we will see below, the jacobian $J$ of a compact Riemann surface $C$ can be interpreted as the moduli space of flat (degree zero) holomorphic line bundles on $C$. In particular, given a flat holomorphic line bundle $\ll$ on $C$ one has a natural class $[\ll]$ on $J$.
\begin{prop} \label{deligneflatform}
Suppose that $\ll, \mm$ are two hermitian line bundles on $\cc$, whose first Chern form vanishes on each of the fibers of $\cc \to B$. In particular $\ll, \mm$ are of degree zero in each of the fibers of $p$. Then the underlying holomorphic line bundles give rise to two sections $[\ll],[\mm]$ of the jacobian bundle $q \colon \jj \to B$, and the equalities
\[ c_1( \langle \ll, \mm \rangle) = -2 \, ([\ll],[\mm])^* \omega_0^\natural \, , \quad  c_1(\langle \ll,\ll \rangle) = -2\, [\ll]^* \omega_0 \]
hold in $A^2(B)$.
\end{prop}
For the proof of Proposition \ref{deligneflatform} it is convenient to recall the notion of Poincar\'e bundle for complex tori. Let $T=V/\Lambda$ be a complex torus, where $V$ is a finite dimensional complex vector space, and $\Lambda$ a lattice in $V$. The dual torus $\hat{T}$ is defined to be the complex torus $\hat{T}=\hat{V}/\hat{\Lambda}$ where $\hat{V}$ is the complex vector space of $\CC$-antilinear forms on $V$ and $\hat{\Lambda}\subset \hat{V}$ consists of those $l \in \hat{V}$ such that $\mathrm{Im} \, l(\Lambda) \subset \zz$. Via the map $l \mapsto \exp(2\pi \sqrt{-1} \, \mathrm{Im} \, l(\cdot))$ the dual torus $\hat{T}$ can be canonically interpreted as the moduli space $\mathrm{Hom}(\Lambda,U(1))$ of flat holomorphic line bundles on $T$ \cite[Proposition 2.4.1]{bl}.

In general \cite[Section 2.2]{bl}, holomorphic line bundles on a complex torus $T$ are determined by so-called Appell-Humbert data which are pairs $(H,\chi)$ consisting of a hermitian form $H$ on $V$ whose imaginary part $\mathrm{Im} \, H$ takes integer values on $\Lambda \times \Lambda$, together with a semicharacter $\chi$ for $H$, i.e. a map $\chi \colon \Lambda \to U(1)$ satisfying the relation $\chi(\lambda+\mu)=\chi(\lambda)\chi(\mu)\exp(\pi \sqrt{-1} \,\mathrm{Im} \, H(\lambda,\mu))$ for all $\lambda,\mu \in \Lambda$. The Poincar\'e bundle $P$ on $T \times \hat{T}$ is defined \cite[Section 2.5]{bl} to be the holomorphic line bundle on $T \times \hat{T}$ given by the Riemann form $((u,u'),(v,v')) \mapsto \overline{v'(u)} + u'(v)$ and associated semi-character $\chi(\lambda,\lambda')=\exp(\pi \sqrt{-1} \, \mathrm{Im} \, \lambda'(\lambda))$. Importantly, one way of viewing $P$ more intrinsically is as a universal flat holomorphic line bundle on $T$, in the following sense: for each $[\ll] \in \hat{T}$, the restriction of $P$ to $T \times \{ [\ll] \}$ is isomorphic to $\ll$ \cite[Section~2.5]{bl}.

As defined above, the Poincar\'e bundle $P$ is equipped with a rigidification at the origin $\epsilon \colon P(0) \isom \CC$. By \cite[Section 3.4]{bl} there is a unique hermitian metric $\|\cdot \|$ on $P$ such that the first Chern form of $(P,\|\cdot\|)$ is translation-invariant, and the rigidification $\epsilon \colon P(0) \isom \CC$ is an isometry, where $\CC$ is endowed with the standard euclidean metric. We will always view $P$ as endowed with this metric.

Now let $C$ be a compact Riemann surface of genus $g \geq 1$ and let $k$ be an integer. Consider the higher jacobian $J_k=\bigwedge^{2k+1} H_\rr/\bigwedge^{2k+1} H_\zz$. We recall that $J_k$ can be naturally viewed as a complex torus with tangent space at the origin identified with the complex vector space $V_k=(\oplus_{p+q=2k+1, p \geq q} H^{p,q})^*$, with $H^{p,q}$ the space of harmonic $(p,q)$-forms on $J_0=J$, the ordinary jacobian of $C$. We have a canonical hermitian form $H_k$ on the complex vector space $V_k$ induced by $M_k$, given by
\[ H_k(u,v)=M_k(\sqrt{-1}\,u,v)+\sqrt{-1}\,M_k(u,v)  \]
for all $u,v \in V_k$. In particular $M_k$ can be viewed as the imaginary part of $H_k$. Let $\hat{J}_k$ be the complex torus dual to $J_k$, and denote by $P_k$ the Poincar\'e bundle on $J_k \times \hat{J}_k$. There is a natural map $\mu_k \colon J_k \to \hat{J}_k$ induced from the linear map $V_k \to \hat{V}_k$ given by sending $u \in V_k$ to $H_k(u,\cdot) \in \hat{V}_k$. For example, in the case $k=0$, as $M_0$ is unimodular on $H_\zz$, the map $\mu_0$ is an isomorphism of complex tori, called the principal polarization of $J_0$. Via $\mu_0$ the jacobian $J_0$ can be canonically interpreted as the moduli space $\mathrm{Hom}(H_\zz,U(1))$ of flat holomorphic line bundles on $C$.

Denote the hermitian line bundle $(\mathrm{id},\mu_k)^*P$ on $J_k \times J_k$ by $P_{\mu,k}$. Let $(\ell_1,\ldots,\ell_{2g})$ be a symplectic basis of $H_\rr$, and let $(\lambda_1,\ldots,\lambda_{2g})$ be the dual basis of $H_\rr^*$.
\begin{lem} \label{c1_poinc1} We have the identity
$ c_1(P_{\mu,k}) = \sum_{1 \leq i_1 < \ldots < i_{2k+1} \leq 2g} \lambda_{i_1}(u) \wedge\ldots \wedge \lambda_{i_{2k+1}}(u) \wedge \lambda_{g+i_1}(v) \wedge \ldots \wedge \lambda_{g+i_{2k+1}}(v)  = 2\,\omega_k^\natural $ of $2$-forms on $J_k \times J_k$.
\end{lem}
\begin{proof} Recall that the map $\mu_k \colon J_k \to \hat{J}_k$ is given by sending $u $ to $H_k(u,\cdot)$. We find that the Riemann form of $P_{\mu,k}$ is given by sending $((u,u'),(v,v'))$ to $\overline{H_k(v',u)}+H_k(u',v)=H_k(u,v')+H_k(u',v)$. As the first Chern form of a holomorphic line bundle on a complex torus can be identified with the imaginary part of its Riemann form, we find the proposition.
\end{proof}
A slight variant is the Poincar\'e bundle $P_0$ on $C \times J$ \cite[Section~11.3]{bl}. This line bundle can be viewed as a universal flat holomorphic line bundle on $C$, that is, for each $[\ll] \in J$ the restriction of $P_0$ to $C \times [\ll]$ is isomorphic to $\ll$. Recall that we can view $(\lambda_1,\ldots,\lambda_{2g})$ as a basis of the complex vector space $\hh$ of harmonic $1$-forms on $C$. From Lemma \ref{c1_poinc1} and the proof of \cite[Proposition 11.3.2]{bl} we find the following.
\begin{lem} \label{c1_poinc2} The equality
$ c_1(P_0) = \sum_{i=1}^{2g} \lambda_i(u) \wedge \lambda_{g+i}(v) $
holds on $C \times J$.
\end{lem}
Let $\tt \to B$ be a complex torus bundle and $\hat{\tt}$ the corresponding family of dual tori. The product family $\tt \times_B \hat{\tt}$ carries a Poincar\'e bundle $\pp$, rigidified at the zero section. The fiberwise canonical metric defines a smooth hermitian metric on $\pp$. This metric on $\pp$ is translation-invariant in all fibers, and the rigidification of $\pp$ along the zero section is an isometry, where the trivial bundle $\CC \times B$ is equipped with the standard euclidean metric.

Let $p \colon \cc \to B$ be a family of compact Riemann surfaces, with associated 
jacobian bundle $q \colon \jj \to B$, viewed as a bundle of complex tori. Let $\hat{\jj}$ be the family of dual tori. By the above, the product family $\jj \times_B \hat{\jj}$ carries a canonically metrised Poincar\'e bundle $\pp$, rigidified at the zero section. Let $\mu \colon \jj \isom \hat{\jj}$ be the canonical principal polarization of $\jj$ coming from the intersection form $M \colon \hh_\rr \otimes \hh_\rr \to \rr$. Write $\pp_\mu = (\mathrm{id},\mu)^* \pp$ on $\jj \times_B \jj$, and write $\pp_0$ for the natural Poincar\'e bundle on $\cc \times_B \jj$. Working locally on $B$, let $(\ell_1,\ldots,\ell_{2g})$ be a symplectic frame of $\hh_\rr$, and $(\lambda_1,\ldots,\lambda_{2g})$ the dual frame of $\hh_\rr^*$.


Lemmas \ref{c1_poinc1} and \ref{c1_poinc2} have the following global counterpart \cite[Proposition 7.3]{hrar}.
\begin{prop} \label{chern} One has $ c_1(\pp_{\mu,k})= \sum_{1 \leq i_1 < \ldots < i_{2k+1} \leq 2g} \lambda_{i_1}(u) \wedge\ldots \wedge \lambda_{i_{2k+1}}(u) \wedge \lambda_{g+i_1}(v) \wedge \ldots \wedge \lambda_{g+i_{2k+1}}(v)  = 2\, \omega_k^\natural $ locally on $\jj_k \times_B \jj_k$.
Likewise, view $(\lambda_1,\ldots,\lambda_{2g})$ as a frame of relative harmonic forms on the fibers of $\cc \to B$. Then the identity $ c_1(\pp_0) = \sum_{i=1}^{2g} \lambda_i(u) \wedge \lambda_{g+i}(v) $ holds locally on $\cc \times_B \jj$.
\end{prop}
\begin{cor} \label{integral} The $2$-forms $2\,\omega_k^\natural$ and $2 \,\omega_k$ have integral cohomology class.
\end{cor}
\begin{proof}[Proof of Proposition \ref{deligneflatform}] The second equality follows directly from the first one as the form $\omega_0^\natural$ restricts to the form $\omega_0$ on the diagonal. In order to prove the first equality, we observe that there is a natural isometry of hermitian line bundles $\langle \ll,\mm \rangle \isom ([\ll],[\mm])^*  \dd$, where $\dd$ is the ``universal Deligne pairing'' on $\jj \times_B \jj$, obtained as follows: let $p_{12},p_{13} \colon \cc \times_B \jj \times_B \jj \to \cc \times_B \jj$ be the projections on the first and second coordinate resp.\! the first and third coordinate. Let $\pp_0$ be the Poincar\'e bundle on $\cc \times_B \jj$. Then the universal Deligne pairing is given as
$ \dd = \langle p_{12}^* \pp_0 , p_{13}^* \pp_0 \rangle$, where the pairing is taken along the projection $p_{23} \colon \cc \times_B \jj \times_B \jj \to \jj \times_B \jj$ on the second and third coordinate. 

We are done once we prove that $c_1(\dd)=-2\,\omega_0^\natural$. We calculate, using Propositions \ref{c1deligne} and \ref{chern}, viewing the $\lambda_i(u)$ as relative harmonic $1$-forms on the family $p_{23}$ of Riemann surfaces,
\[ \begin{split}
c_1(\dd) &= \int_{p_{23}} \left( \sum_{i=1}^{2g} \lambda_i(u) \wedge \lambda_{g+i}(v) \right) \wedge \left( \sum_{i=1}^{2g} \lambda_i(u) \wedge \lambda_{g+i}(w) \right) \\
& = \int_{p_{23}} \sum_{i=1}^{2g} \lambda_i(u) \wedge \lambda_{g+i}(v) \wedge \lambda_{g+i}(u) \wedge \lambda_{2g+i}(w) \\
& = - \sum_{i=1}^{2g} \int_{p_{23}} \lambda_i(u) \wedge \lambda_{g+i}(u) \wedge \lambda_{g+i}(v) \wedge \lambda_{2g+i}(w) \\
& = - \sum_{i=1}^{2g} \lambda_i(v) \wedge \lambda_{g+i}(w) \\
&= -2\, \omega_0^\natural \, ,
\end{split} \]
and we are done.
\end{proof}
\begin{remark}  One can actually prove a canonical isomorphism
$ \dd \isom \pp_\lambda^{\otimes -1} $
of hermitian line bundles on $\jj \times_B \jj$.
\end{remark}

\section{Proof of identities (K1) and (K3)} \label{K1K3}

In this section we give proofs of identities (K1) and (K3). A proof of identity (K2) will be given in Section~\ref{K2}. We will make extensive use of the Deligne pairing, and especially Proposition \ref{deligneflatform}, in the various situations at hand.

Let $\jj$ be the universal jacobian over $\mm_g$. Recall that identity (K1) deals with the map $ \kappa \colon \cc_g \to \jj$ given by sending a pointed surface $(C,x)$ to $(J(C),[(2g-2)x-K_C])$, where $J(C)$ denotes the jacobain of $C$ and $K_C$ is a canonical divisor on $C$. We abbreviate the relative tangent bundle of $p \colon \cc_g \to \mm_g$ by $T$.
\begin{proof}[Proof of equality (K1)]
Denote by $\Delta$ the diagonal of $\cc_g^2$. From Proposition \ref{deligneflatform} we obtain the equality
\[ 2\,\kappa^* \omega_0 = -c_1\left(\langle \oo((2g-2)\Delta) \otimes p_2^*T, \oo((2g-2)\Delta) \otimes p_2^* T \rangle \right) \]
of $2$-forms on $\cc_g$. We expand the right hand side using the biadditivity of the Deligne pairing. First of all we have
\[ c_1(\langle \oo(\Delta), \oo(\Delta) \rangle ) = \int_{p_1} h^2 = e^A \]
by Proposition \ref{eA_as_fiberintegral}. Next we have
\[ c_1(\langle \oo(\Delta),p_2^*T \rangle)) = c_1(\langle p_2^*T,\oo(\Delta) \rangle))= c_1(T)=e^A \quad \mathrm{and} \quad c_1(\langle p_2^*T,p_2^*T \rangle) = e_1^A \, . \]
By combining we obtain
\[  2\,\kappa^* \omega_0 = -2g(2g-2) \, e^A - e_1^A \, , \]
which is the required identity.
\end{proof}
Recall that identity (K3) deals with the map $\delta \colon \cc_g^2 \to \jj$   given by sending a $2$-pointed Riemann surface $(C,x,y)$ to $(J(C),[y-x])$.
\begin{proof}[Proof of equality (K3)]
Let $x,y \colon \cc_g^2 \to \cc_g^3$ be the two canonical sections of the universal Riemann surface $p_{12} \colon \cc_g^3 \to \cc_g^2$ over $\cc_g^2$.   From Proposition \ref{deligneflatform} we obtain the equality
\[ 2 \,\delta^* \omega_0 = -c_1 \left( \langle \oo(y-x),\oo(y-x) \rangle \right) \]
of $2$-forms on $\cc_g^2$. Again we expand the right hand side using the biadditivity of the Deligne pairing. From equation (\ref{Deligne}) we obtain that
\[ c_1(\langle \oo(x),\oo(x) \rangle) = p_1^*e^A \quad \mathrm{and} \quad  c_1(\langle \oo(y),\oo(y) \rangle) = p_2^*e^A \, . \]
Further we have an isometry $ \langle \oo(x),\oo(y) \rangle \isom \oo(\Delta) $
where $\Delta$ is the diagonal on $\cc_g^2$, so that $ c_1(\langle \oo(x),\oo(y) \rangle ) = h$. By combining we obtain
\[ 2 \,\delta^* \omega_0 = 2 \,h - p_1^* e^A - p_2^* e^A \, , \]
which is the required identity.
\end{proof}
\begin{remark} \label{eAwithout} Using a variant of the above arguments one can deduce from (K3) the identity
\[ 4 \int_{p_1} (\delta^* \omega_0)^2 = \int_{p_1} (2\,h - p_1^*e^A - p_2^*e^A)^2 = -4g \, e^A + e_1^A \, . \]
Combining this identity with (K1), we find the expression
\[ e^A = -\frac{1}{4g^2}\left( 4 \int_{p_1} (\delta^* \omega_0)^2 + 2\kappa^*\omega_0 \right) \]
for $e^A$. In particular, one can write $e^A$ in terms of $\delta, \kappa$ and $\omega_0$ only, without mentioning the Arakelov-Green's function $G \colon \cc_g^2 \to \rr$.
\end{remark}

\section{Pointed harmonic volume} \label{pointed}

In this section we recall the Harris-Pulte harmonic volume \cite{harris} \cite{pu} of a pointed Riemann surface. Let $(C,x)$ be a pointed compact Riemann surface of genus $g \geq 1$ and as above write $H_\zz = H_1(C,\zz)$ and $H_\rr=H_\zz \otimes \rr$. Let $H^*$ denote the Hodge structure dual to $H_\zz$; recall that $H^*_\CC = H^* \otimes \CC$ is naturally identified with the space of harmonic $1$-forms $\hh$ on $C$. The first intermediate jacobian $J_1 = \largewedge^3 H_\rr/ \largewedge^3 H_\zz \simeq (\largewedge^3 H_\zz ) \otimes (\rr/\zz)$ can alternatively be written as $J_1\simeq\mathrm{Hom}(\largewedge^3 H^*, \rr/\zz)$.

Denote by $K$ the kernel of the canonical map $M^* \colon H^*  \otimes H^* \to \zz$ induced by the intersection pairing $M$. The harmonic volume  of $(C,x)$ is defined to be the element $I(x)$ of the torus $J' = \mathrm{Hom}(K \otimes H^*,\rr/\zz)$ that as a functional on $K \otimes H^*$ sends the element $\sum_i \left(\varphi_1^{(i)} \otimes \varphi_2^{(i)} \right) \otimes \varphi_3$ to the real number $\int_{\gamma_3} \left( \sum_i \tilde{\varphi}_1^{(i)} \tilde{\varphi}_2^{(i)} + \eta \right)$ modulo $\zz$, where $\tilde{\varphi}_j^{(i)}$ is the harmonic $1$-form representing $\varphi_j^{(i)}$, $\eta$ is a (unique) $1$-form orthogonal to all closed $1$-forms such that $\sum_i \varphi_1^{(i)} \wedge \varphi_2^{(i)} + d\eta = 0$ holds, and $\gamma_3$ is a loop with base point $x$ whose class in $H_1(C,\zz)$ is the Poincar\'e dual of $\varphi_3$. Here, the integral $\int_{\gamma_3} \tilde{\varphi}_1^{(i)} \tilde{\varphi}_2^{(i)}$ is an iterated integral in the sense of K.T. Chen \cite{chen}. Varying the point $x$ on $C$, one obtains a natural $C^\infty$ map $I \colon C \to J'$. 

Using a symplectic basis one checks \cite[Lemma 4.7]{pu} that the wedge product $K \otimes H^* \to \largewedge^3 H^*$ is surjective. In particular, we have a natural injective map $J_1 \rightarrowtail J'$. As it turns out \cite{pu}, the pointed harmonic volume $I(x) \in J'$ of $(C,x)$ in fact lies in $J_1$.

A fundamental result is then the following.
\begin{prop} \label{pulte} Let $\gamma \colon C \to J_1$ be the map associating to $x \in C$ the Griffiths Abel-Jacobi class of the Ceresa cycle $C_x - C_x^{-}$, and let $I \colon C \to J_1 \subset J'$ be Harris-Pulte's pointed harmonic volume. Then the identity $\gamma =2\,I$ holds.
\end{prop}
\begin{proof} See \cite[Theorem 4.10]{pu}.
\end{proof}
We note that by wedging with the dual of the symplectic form in $\largewedge^2 H^*$, we obtain a natural torus embedding $i \colon J \rightarrowtail J_1$, where $J=J_0$ is the ordinary jacobian of $C$. The results in \cite{harris} \cite{pu} show that for two points $x,y \in C$, the difference $I(x)-I(y)$ lies in $J$, and can be identified with the Abel-Jacobi class of the divisor $y-x$. In particular, taking the image under the quotient mapping $J_1 \to J_1/J$ yields an element $\bar{I}$ of $J_1/J$ which is independent of the base point $x$.

There is a natural contraction map $c \colon \largewedge^3 H \to H$ given by sending $a \wedge b \wedge c \in \largewedge^3 H$ to $M(a,b)c + M(b,c)a + M(c,a)b$. This induces a natural map $c \colon J_1 \to J$. For instance, it is not difficult to see that $ci = (g-1)\mathrm{id}_J$. Let $\kappa \colon C \to J$ be the map given by $x \mapsto [(2g-2)x-K_C]$, where $K_C$ is a canonical divisor on $C$. 
\begin{prop} \label{cgamma} The identity $c\gamma=\kappa$ holds.
\end{prop}
\begin{proof} See \cite[Corollary 6.7]{pu}.
\end{proof}
As usual, the above constructions and results can be globalized over the universal family of Riemann surfaces $\cc_g$. Let $\jj_1$ be the first intermediate jacobian fibration over $\cc_g$ as in Section \ref{prelims}. Varying the pointed Riemann surface $(C,x)$, the pointed harmonic volume gives a natural section $I \colon \cc_g \to \jj_1$. The first variation $dI \in A^1(\cc_g,\largewedge^3 \hh_\rr)$ is a twisted $1$-form that represents the extended first Johnson homomorphism on the mapping class group of a pointed oriented surface \cite{magnus} \cite{extension}. 

Let $\gamma \colon \cc_g \to \jj_1$ be the section of $\jj_1$ obtained by applying Griffiths's Abel-Jacobi construction to the universal Ceresa cycle $\ZZ$ inside the universal Jacobian bundle $q \colon \jj \to \cc_g$. We obtain from Proposition \ref{pulte} that $\gamma = 2\,I$. Further, the contraction $c$ introduced above globalizes to give a natural map of torus bundles $c \colon \jj_1 \to \jj$. Let $\kappa \colon \cc_g \to \jj$ be the usual map given by sending $(C,x)$ to $(J(C),[(2g-2)x-K_C])$, then Proposition \ref{cgamma} implies that $c\gamma=\kappa$ as maps from $\cc_g $ to $\jj$. 


\section{Proof of identity (K2)} \label{K2}

Let $q \colon \jj \to \cc_g$ denote the universal Jacobian bundle over $\cc_g$. We continue with the notation from Section \ref{pointed}. Put $\psi = 2 \,\omega_0^\natural$ on $\jj \times \jj$. Here and below, products are fiber products over $\cc_g$. Working locally on $\cc_g$, let $(\ell_1,\ldots,\ell_{2g})$ be a symplectic frame of $\hh_\rr$, and $(\lambda_1,\ldots,\lambda_{2g})$ the dual frame of $\hh_\rr^*$. Then by Proposition \ref{chern} we locally have the equality $ \psi = \sum_{i=1}^{2g} \lambda_i(u) \wedge \lambda_{g+i}(v)$, where we still use the convention $\lambda_{2g+i}=-\lambda_i$ for $i=1,\ldots,g$. The $6$-form $\psi^3$ is a multiple of $\omega^\natural_1$; more precisely we have, computing locally,
\begin{equation} \label{psi^3} \begin{split}
\psi^3 &=  6 \sum_{1 \leq i <j<k \leq 2g} \lambda_i(u) \wedge\lambda_{g+i}(v) \wedge \lambda_j(u) \wedge \lambda_{g+j}(v) \wedge \lambda_k(u) \wedge \lambda_{g+k}(v) \\
 & = -6 \sum_{1 \leq i <j<k\leq 2g} \lambda_i(u) \wedge \lambda_j(u)\wedge\lambda_k(u) \wedge \lambda_{g+i}(v) \wedge \lambda_{g+j}(v) \wedge \lambda_{g+k}(v) \\
 & = - 12 \, \omega^\natural_1 \, .
\end{split}
\end{equation}
Let $Q \colon \jj \times \jj \to \cc_g$ be the structure map. Denote by $\cc$ the universal family of Riemann surfaces embedded in $\jj$ via the universal Abel-Jacobi map $\delta \colon \cc_g^2 \to \jj$.
Let $x,y,z \colon \cc_g^3 \to \cc_g^4$ be the canonical sections of the universal Riemann surface $p_{123} \colon \cc_g^4 \to \cc_g^3$ over $\cc_g^3$.
Put
\[ \delta_1 \colon \cc_g^3 \to \jj \times \jj \, , \quad (C,(x,y,z)) \mapsto (J(C),[y-x],[z-x]) \, . \]
Then $\delta_1$ is an embedding over $\cc_g$ and the image of $\delta_1$ in $\jj \times \jj$ equals $\cc \times \cc$. A small computation yields that for forms $\eta$ on $\jj \times \jj$ invariant under inversion the identity
\begin{equation} \label{times4} \int_Q \eta \, \delta_{\ZZ \times \ZZ} = 4 \int_Q \eta \, \delta_{\cc \times \cc}
\end{equation}
holds. Let $p_{(1)} \colon \cc_g^3 \to \cc_g$ be the projection on the first coordinate.
We calculate
\begin{eqnarray*} 48\,I^*\omega_1 = 48 \, M_1\,(\d I)^{\otimes 2} = 12\,M_1\,(\d \gamma)^{\otimes 2} && \textrm{by Proposition \ref{pulte}} \\
  = 12 \int_Q \omega_1^\natural \wedge \delta_{\ZZ \times \ZZ} \, && \textrm{by Proposition \ref{fiberint}} \\
  = - \int_Q \psi^3 \wedge \delta_{\ZZ \times \ZZ} \, && \textrm{by equation (\ref{psi^3})} \\
  = - 4 \int_Q \psi^3 \wedge \delta_{\cc \times \cc} && \textrm{by equation (\ref{times4})}  \\
  = - 4 \int_{p_{(1)}} (\delta_1^*\psi)^3 \quad && \textrm{by pullback.}  \\
\end{eqnarray*}
This leads to the fiber integral expression
\begin{equation} \label{end} 12\,I^*\omega_1 = - \int_{p_{(1)}} (\delta_1^* \psi)^3 \, .
\end{equation}
\begin{proof}[Proof of equality (K2)]
Let $F \colon \cc_g^2 \to \mm_g$ be the canonical map. By Proposition \ref{a_gfiberint} we have
\[ e_1^A - 2\sqrt{-1} \,\deldelbar \, a_g = \int_F h^3 \, . \]
By equation (\ref{end}) the proof of (K2) would be finished once we prove that
\begin{equation} \label{todo} - \int_{p_{(1)}} (\delta_1^* \psi)^3 = -6g \, e^A + \int_F h^3 \, . 
\end{equation}
We will prove this identity by an explicit calculation, inspired by the proof of Theorem 2.5.1 in \cite{zh}. First, by Proposition \ref{deligneflatform} we have an equality
\[ -\delta_1^* \psi = - 2 \, \delta_1^* \omega_0^\natural =  c_1( \langle \oo(y-x),\oo(z-x) \rangle) \]
of $2$-forms on $\cc_g^3$. Next, let $p_{ij} \colon \cc_g^3 \to \cc_g^2$ for $1\leq i < j \leq 3$ be the projection on the $i$-th and $j$-th coordinate. Let $\Delta$ be the diagonal on $\cc_g^2$. Then by biadditivity of the Deligne pairing we have an isometry
\[ \langle \oo(y-x),\oo(z-x) \rangle \isom p_{23}^* \oo(\Delta) \otimes p_{12}^* \oo(-\Delta) \otimes p_{13}^* \oo(-\Delta) \otimes p_{(1)}^* (\langle \Delta, \Delta \rangle) \]
and hence, upon taking first Chern forms on left and right hand side, the equality 
\[ -\delta_1^* \psi = p_{23}^* h - p_{12}^* h - p_{13}^*h + p_{(1)}^* e^A  \]
by Proposition \ref{eA_as_fiberintegral}. By Lemma \ref{selfinters} we have
\[ \int_{p_{(1)}} (p_{23}^*h - p_{12}^*h - p_{13}^*h)^2 = -2g \]
and this gives
\[ \begin{split}
-\int_{p_{(1)}} (\delta_1^* \psi)^3 = &  \int_{p_{(1)}} (p_{23}^* h - p_{12}^* h - p_{13}^*h + p_{(1)}^* e^A)^3 \\
 = & \int_{p_{(1)}} (p_{23}^*h - p_{12}^*h - p_{13}^*h)^3 + 3 \, e^A \int_{p_{(1)}} (p_{23}^*h - p_{12}^*h - p_{13}^*h)^2 \\
 = & \int_{p_{(1)}} (p_{23}^*h - p_{12}^*h - p_{13}^*h)^3 - 6g \, e^A \\
 = &  \int_{p_{(1)}} (p_{23}^*h)^3 - 3 \int_{p_{(1)}} (p_{23}^*h)^2 (p_{12}^*h + p_{13}^*h) +
    3 \int_{p_{(1)}} p_{23}^*h (p_{12}^*h+p_{13}^*h)^2 \\
    & - \int_{p_{(1)}} (p_{12}^*h + p_{13}^*h)^3 -6g \, e^A \, . \\
\end{split}
\]
Note that $\int_{p_{(1)}} (p_{23}^*h )^3 = \int_F h^3 $. We are done once we prove that 
\[ - 3 \int_{p_{(1)}} (p_{23}^*h)^2 (p_{12}^*h + p_{13}^*h) +
    3 \int_{p_{(1)}} p_{23}^*h (p_{12}^*h+p_{13}^*h)^2  
      - \int_{p_{(1)}} (p_{12}^*h + p_{13}^*h)^3 =0 \, . \]
Let $p_1 \colon \cc_g^2 \to \cc_g$ be the projection on the first coordinate, so that $p_1 \,p_{12}=p_1\,p_{13}=p_{(1)}$. As by Lemma \ref{selfinters} the restrictions of $p_{23}^*h\,p_{12}^*h$, $p_{23}^*h\,p_{13}^*h$ and $p_{12}^*h\,p_{13}^*h$ to a fiber of $p_{(1)}$ are equal, we have
\[
- 3 \int_{p_{(1)}} (p_{23}^*h)^2 (p_{12}^*h + p_{13}^*h) + 6 \int_{p_{(1)}} p_{23}^*hp_{12}^*h p_{13}^*h = 0 \, . \]
Next, by the projection formula we have
\[ \int_{p_{(1)}} p_{23}^*h (p_{12}^*h)^2 = 
\int_{p_{(1)}} p_{23}^*h (p_{13}^*h)^2 = 
\int_{p_{(1)}} (p_{12}^*h)^2 p_{13}^*h =
\int_{p_{(1)}} (p_{13}^*h)^2 p_{12}^*h =
\int_{p_1} h^2  \, . \]
Finally we have
\[ \int_{p_{(1)}} (p_{12}^*h)^3 = \int_{p_{(1)}} (p_{13}^*h)^3 = 0 \, ,  \]
again by the projection formula.
\end{proof}
We will now prove Theorem \ref{secondidentity}. Recall from equation (\ref{e_1^J}) that we put
\[ e_1^J = \frac{1}{2g+1} (-6\,\kappa^*\omega_0 + 12\, (2g-2)I^*\omega_1 ) \, .\]
The following identity follows immediately from Theorem \ref{main}.
\begin{prop} \label{AandJ} The equality
\[
e_1^J - e_1^A = \frac{-2(2g-2)\sqrt{-1}}{2g+1} \deldelbar \, a_g  \]
holds on $\cc_g$.
\end{prop}
\begin{proof}[Proof of Theorem \ref{secondidentity}] By the projection formula, as $e^A-e^J$ is pulled back from $\mm_g$, we have
\[ \begin{split} e_1^F - e_1^A & = -\int_p (e^A - e^J)(e^A + e^J) = 2(2g-2) (e^A-e^J) \\ & = \frac{-2(2g-2)\sqrt{-1}}{g(2g+1)} \, \deldelbar \, a_g  \end{split} \]
by Theorem \ref{kaw}. Using Proposition \ref{AandJ} we find
\[ e_1^F-e_1^J = (e_1^F-e_1^A)-(e_1^J-e_1^A) = \frac{-2(2g-2)^2\sqrt{-1}}{2g(2g+1)} \, \deldelbar \, a_g \]
and we are done.
\end{proof}
\begin{remark} Let $p_1,p_2 \colon \cc_g^2 \to \cc_g$ denote the projections on the first and second coordinate, respectively. Write $k_0$ for the class of $d\delta$ in $H^1(\cc_g^2,\hh_\zz)$. We mention that this class has been studied extensively from the point of view of cohomology of the mapping class groups by Morita \cite{familiesII}. Following \cite{km} we define for each $i,j \in \zz$ the cohomology class
\[ m_{i,j} = \int_{p_1} (p_2^*e)^i k_0^j \in H^{2i+j-2}(\cc_g,\largewedge^j \hh_\zz) \, , \]
called the \emph{twisted} Miller-Morita-Mumford class of type $(i,j)$. Note that the usual Miller-Morita-Mumford class $e_i \in H^{2i}(\cc_g,\zz)$ for $i \in \zz$ is given as the special case $e_i = m_{i+1,0}$. Another interesting special case is $m_{0,3}=\int_{p_1} k_0^3$, which according to \cite[Proposition 7.1]{km} is equal to $-6\,\tilde{k}$, where $\tilde{k}$ is the class in $H^1(\cc_g,\largewedge^3 \hh_\zz)$ determined by the first variation $dI$ of the pointed harmonic volume.

As is proved in \cite{km}, by contracting the coefficients of powers of the $m_{i,j}$ using the intersection pairing $M$ one always obtains a class in the tautological algebra $\qq[e,e_i]$ of $H^*(\cc_g,\zz)$. 
Given that $I^*\phi_1=M(\tilde{k}^2)$ and that $\kappa^*\phi_0$ can be obtained by further contracting $I^*\phi_1$ using $M$ by Proposition \ref{cgamma}, equations (M1) and (M2) yield examples of this general rule. 

Taking $e^A$ to represent $e$, and $d\delta$ to represent $k_0$, we obtain canonical twisted $j$-forms representing the classes $m_{i,j}$, and hence any class with trivial coefficients obtained by contracting the coefficients of powers of the $m_{i,j}$ is represented by a canonical form. As a result, some elements of $\qq[e,e_i]$ may be represented by various ``canonical'' forms in this way, and it seems of interest to study the resulting exact forms on $\cc_g$ obtained by taking the differences of such ``canonical forms'' realizing a given cohomology class. 

Using a variant of Proposition \ref{fiberint} one can show that the canonical twisted $1$-form $\int_{p_1} (d\delta)^{\otimes 3}$ representing $m_{0,3}$ is precisely $-6 \, dI$. Combined with Remark \ref{eAwithout} we find that in cohomological degree two (i.e. for $e,e_1$), by applying Kawazumi-Morita's recipe one obtains at least the canonical forms $e^A, e^J$ resp. $e_1^A, e_1^J$ representing $e$ resp. $e_1$, and Kawazumi's results \cite[Theorems~0.1, 0.2]{kawpr} show that their differences are rational multiples of the second variation $\deldelbar \, a_g$ of his invariant. 
\end{remark}

\section{Faltings delta- and Hain-Reed beta-invariant} \label{hainreedfaltings}

In the introduction to his paper \cite{kawpr} Kawazumi asked how the invariant $a_g$ is related with two other well known conformal invariants, namely the Faltings delta- and the Hain-Reed beta-invariant. In this section we will give a precise answer to this question. The result has appeared earlier as \cite[Theorem~1.4]{djsecond}, but the method of proof below will be more direct.

We start by introducing the Faltings delta-invariant. Assume that $g \geq 2$. Let $\ll$ be the determinant of the Hodge bundle of $p \colon \cc_g \to \mm_g$, and let $\omega_{\mathrm{Hdg}}$ denote the first Chern form of the $L^2$-metric (Petersson norm) on $\ll$. As is well known, by the Atiyah-Singer index theorem we have the equality $e_1=12\,\ll$ in $\H^2(\mm_g,\zz)$. 

Let $C$ be a compact Riemann surface of genus $g \geq 2$. The Faltings delta-invariant \cite{fa} \cite{we} $\delta_g$ of $C$ can be succinctly described as the analytic torsion \cite{rs} for the Arakelov metric on the holomorphic cotangent bundle $T^*C$, that is we have
\[ \delta_g = -6 \log \frac{ \det' \Delta }{\mathrm{vol} \, C} \, , \]
where $\det' \Delta$ is the zeta regularized determinant of the Laplace operator
with respect to the Arakelov metric, and $\mathrm{vol}\, C$ the associated volume of $C$. By the local Grothendieck-Riemann-Roch theorem \cite[Th\'eor\`eme~2.2]{bost} \cite[Th\'eor\`eme~11.4]{de}, which is a refined version of the Atiyah-Singer index theorem and Quillen's formula \cite[Theorem~1]{qu}, the delta-invariant solves the differential equation
\begin{equation} \label{noether} e_1^A-12 \, \omega_{\mathrm{Hdg}} = \frac{\deldelbar}{\pi\sqrt{-1}} \,\delta_g 
\end{equation}
on $\mm_g$. Recall from equation (\ref{e_1^J}) that we put
\[ e_1^J = \frac{1}{2g+1} (-6\,\kappa^*\omega_0 + 12\, (2g-2)I^*\omega_1 ) \, .\]
As was remarked before, by (M1)--(M2) the form $e_1^J$ represents the class $e_1$. Its significance can be described as follows. Wedging with the dual $\omega_0^*$ of the symplectic form yields a natural injective map $i \colon \jj_0 \to \jj_1$, and the discussion in Section \ref{pointed} shows that composing the Ceresa cycle section $\gamma \colon \cc_g \to \jj_1$ with the projection $q \colon \jj_1 \to \jj_1/\jj_0$ gives rise to a natural section $\nu \colon \mm_g \to \jj_1/\jj_0$. 

Following \cite[Section~5]{hrgeom}, the quotient $\jj_1/\jj_0$ carries a natural parallel closed $2$-form $\omega_{1,0}$, and hence a natural corresponding degree two cohomology class $\phi_{1,0}$. They are given as follows.  Consider again the contraction $c \colon \jj_1 \to \jj_0$ sending $a \wedge b\wedge c$ to $M(a,b)c + M(b,c)a + M(c,a)b$. Then as $ci=(g-1)\mathrm{id}_\jj$ we have a natural well-defined map of torus bundles $j \colon \jj_1/\jj_0 \to \jj_1$ given by sending $q(a \wedge b \wedge c)$ to $a \wedge b \wedge c - \omega^*_0 \wedge c(a \wedge b \wedge c)/(g-1)$. The closed $2$-form $\omega_{1,0}$ is defined to be the pullback $j^* \omega_1$ of $\omega_1$, and $\phi_{1,0} \in H^2(\jj_1/\jj_0)$ is to be its class in cohomology. Note that by Corollary \ref{integral} the class $2 \, \omega_{1,0}$ is integral. 

Morita in \cite{molinear} has computed the cohomology class $2 \,\nu^* \phi_{1,0}$ in $H^2(\mm_g,\zz)$. 
\begin{thm} \label{moritaMg} (Morita, \cite[Theorem 7]{hrgeom} \cite[Theorem 5.8]{molinear}) Assume that $g \geq 2$. The equality of cohomology classes
\[ 2 \, \nu^* \phi_{1,0} = (8g+4) \, \ll \]
holds in $\H^2(\mm_g,\zz)$.
\end{thm}
It turns out that we can be a bit more precise: one can write the pullback $2$-form $\nu^* \omega_{1,0}$ as a rational multiple of the form $e_1^J$.  
\begin{prop} \label{between} We have an equality
\[ 6 \, \nu^* \omega_{1,0} = (2g+1) \, e_1^J \]
of $2$-forms on $\mm_g$. 
\end{prop}
\begin{proof} By \cite[Proposition 18]{hrgeom} or \cite[Lemma 5.1]{kawpr} we have
the equality
\[ q^* \omega_{1,0} = (g-1)\,\omega_1 - c^* \omega_0   \]
of $2$-forms on $\jj_1$. By Proposition \ref{cgamma} we have $c\gamma=\kappa$. The proposition follows by pulling back along $\gamma$.
\end{proof}
We thus reobtain Morita's result Theorem \ref{moritaMg}: Proposition \ref{between} gives at least that $6\,\nu^*\phi_{1,0}=3(8g+4)\,\ll$ by taking cohomology classes. But $\H^2(\mm_g,\zz)$ is infinite cyclic if $g \geq 3$, and cyclic of order $10$ if $g=2$ \cite{ha1} \cite{ha2}.  

We recall that if $g \geq 3$, for the orbifold $\mm_g$ one has 
\[ H^0(\mm_g,\oo)=\CC \, , \quad H^1(\mm_g,\CC)=H^1(\mm_g,\oo)=0 \, . \] 
The first equality is shown by considering the closure of $\mm_g$ in the Satake compactification of the moduli space of principally polarized abelian varieties, the second set of equalities follows from Harer's celebrated results \cite{ha1} \cite{ha2} on the cohomology of the mapping class group. We infer from this \cite[Lemma 8.1]{kawhb} that if a real $(1,1)$-form $\eta$ on $\mm_g$ is $d$-exact, then there is a real valued function $f \in C^\infty(\mm_g,\rr)$ such that $\eta = \sqrt{-1} \deldelbar f$. Moreover, such a function $f$ is unique up to a constant.

Consider then the real $(1,1)$-form $\eta = 2\, \nu^* \omega_{1,0} - (8g+4) \, \omega_{\mathrm{Hdg}} $ on $\mm_g$. As was observed by Hain and Reeed in \cite{hrar},  Morita's result Theorem \ref{moritaMg} implies that the form $\eta$ is $d$-exact, and hence there exists a function $\beta_g \in C^\infty(\mm_g,\rr)$, unique up to a constant, solving the $\deldelbar$-equation
\begin{equation} \label{defbeta} \frac{\deldelbar}{\pi\sqrt{-1}} \, \beta_g = 2\, \nu^* \omega_{1,0} - (8g+4) \, \omega_{\mathrm{Hdg}}  \, .
\end{equation}
We call $\beta_g \in C^\infty(\mm_g,\rr)/\rr$ the Hain-Reed beta-invariant. Our results so far easily show that $\beta_g$ can be expressed as a rational linear combination of Faltings's delta-invariant and Kawazumi's $a_g$-invariant. 
\begin{thm} \label{beta} (cf. \cite[Theorem~1.4]{djsecond}) We have
\[ \beta_g = \frac{1}{3}\left( (2g-2) \, \varphi_g + (2g+1) \, \delta_g \right) \]
as a natural representative in $C^\infty(\mm_g)$ of Hain-Reed's beta-invariant.
\end{thm}
\begin{proof} Combining equation (\ref{noether}) and Proposition \ref{AandJ} we find, with $\varphi_g = 2\pi \, a_g$, the relation
\[ e_1^J - 12 \, \omega_{\mathrm{Hdg}} = \frac{\deldelbar}{\pi\sqrt{-1}} \left( \frac{2g-2}{2g+1} \varphi_g + \delta_g \right) \, . \]
Then with Proposition \ref{between} and equation (\ref{defbeta}) we obtain the result.
\end{proof}
In their paper \cite{hrar} Hain and Reed determined the asymptotic behavior of their invariant $\beta_g$  towards the boundary of $\mm_g$ in the Deligne-Mumford compactification \cite[Theorem~1]{hrar}. The proof in \cite{hrar} uses topological arguments, in particular the Johnson homomorphisms. 

Using Theorem \ref{beta}, an alternative approach can be given. First of all, J. Jorgenson \cite{jo} and R. Wentworth \cite{we} determined the asymptotics of $\delta_g$.
\begin{thm} \label{asymptdelta} Let $D$ denote the unit disk in $\CC$ and let $X \to D$ be a proper holomorphic map which restricts to a family of compact Riemann surfaces of genus $g \geq 2$ over $D \setminus \{0\}$, and where the fiber $X_0$ over $0$ is a stable complex curve.
\begin{itemize}
\item[(i)] If $X_0$ is irreducible with only one node, then
\[ \delta_g(X_t) \sim -\frac{4g-1}{3g} \log |t| - 6 \log \log (1/|t|)   \]
as $t \to 0$.
\item[(ii)] If $X_0$ is reducible with one node and its components have genera $i$ and $g-i$ then
\[  \delta_g(X_t) \sim -\frac{4i(g-i)}{g} \log |t| \]
as $t \to 0$.
\end{itemize}
Here, if $f,g$ are two functions on the punctured unit disk, the notation $f \sim g$ denotes that $f-g$ is bounded as $t\to 0$. 
\end{thm} 
Next, in \cite{djas} we proved the following theorem. 
\begin{thm} \label{asymptphi} Take the assumptions of the previous theorem.
\begin{itemize}
\item[(i)] If $X_0$ is irreducible with only one node, then
\[ \varphi_g(X_t) \sim -\frac{g-1}{6g} \log |t|  \]
as $t \to 0$.
\item[(ii)] If $X_0$ is reducible with one node and its components have genera $i$ and $g-i$ then
\[ \varphi_g(X_t) \sim -\frac{2i(g-i)}{g} \log |t| \]
as $t \to 0$.
\end{itemize}
\end{thm}
Combining this with Theorem \ref{beta} we reobtain the main result \cite[Theorem 1]{hrar} of Hain and Reed. 
\begin{cor} \label{asymptbeta}
Let $D$ denote the unit disk in $\CC$ and let $X \to D$ be a proper holomorphic map which is a family of compact Riemann surfaces of genus $g \geq 2$ over $D \setminus \{0\}$, and where the fiber $X_0$ over $0$ is a stable complex curve.
\begin{itemize}
\item[(i)] If $X_0$ is irreducible with only one node, then
\[ \beta_g(X_t) \sim -g \log |t| - (4g+2) \log \log (1/|t|) \]
as $t \to 0$.
\item[(ii)] If $X_0$ is reducible with one node and its components have genera $i$ and $g-i$ then
\[ \beta_g(X_t) \sim -4i(g-i) \log |t| \]
as $t \to 0$.
\end{itemize}
\end{cor}
We finish by concentrating on the hyperelliptic case. Let $\hh_g \subset \mm_g$ denote the moduli space of hyperelliptic Riemann surfaces of genus $g$. An important first observation is that Kawazumi's $2$-form $e_1^J$ vanishes identically on $\hh_g$.
\begin{prop} \label{useful} Let $\hh'_g \subset \cc_g$ be the moduli space of hyperelliptic Riemann surfaces of genus $g$ marked with a Weierstrass point. Note that we have a topological covering space $\hh'_g \to \hh_g$ of degree $2g+2$. The restrictions of both $\kappa^*\omega_0$ and $I^*\omega_1$ to $\hh'_g \subset \cc_g$ are zero. In particular, the $2$-form $e_1^J$ (see equation (\ref{e_1^J})) vanishes on $\hh_g \subset \mm_g$.
\end{prop}
\begin{proof} Let $C$ be a hyperelliptic Riemann surface of genus $g \geq 2$ and $w$ a Weierstrass point on $C$. We have that $(2g-2)w$ is a canonical divisor on $C$ and it follows that $\kappa(C,w)$ is zero. Let $C_w$ be the Abel-Jacobi image of $C$ into $J(C)$ using $w$ as a basepoint. We have that $[-1]$ then acts as $\iota$ on $C_w$ and hence the Ceresa cycle $C_w-C_w^-$ is zero as well. As the forms $\omega_k$ vanish along the zero section, we find that the restrictions of both $\kappa^*\omega_0$ and $I^*\omega_1=(1/4)\,\gamma^*\omega_1$ to $\hh'_g$ are zero as required. 
\end{proof}
In fact, the above argument shows that the hyperelliptic pointed harmonic volume $I$ is locally constant on $\hh'_g$. An explicit description of $I$ for Weierstrass-pointed hyperelliptic Riemann surfaces is given by Y. Tadokoro in \cite{tad}. 

As $e_1^J$ vanishes identically on $\hh_g$, we find from Proposition \ref{between} and equation (\ref{defbeta}) that
\[ \frac{\deldelbar}{\pi \sqrt{-1}} \, \beta_g = -(8g+4) \, \omega_{\mathrm{Hdg}}  \]
on $\hh_g$. This relation suggests that the restriction of $\beta_g$ to $\hh_g$ may be expressed as a rational multiple of the logarithm of the Petersson norm of a suitable nowhere vanishing modular form on $\hh_g$. In \cite{djsecond} we have proved this to be true indeed. Write $n = {2g \choose g+1}$, then one has
\[ \beta_g|_{\hh_g} = -(8g+4)g\log(2\pi) - (g/n) \log \|\Delta_g \|_{\mathrm{Pet}} \, , \]
where $\Delta_g$ is the so-called discriminant modular form (a product of Thetanullwerte) of weight $(8g+4)n/g$ on $\hh_g$. Combining with Theorem \ref{beta} we obtain the closed expression
\[ 2\pi \, (2g-2) a_g =  -8(2g+1)g \log(2\pi) - 3(g/n) \log\|\Delta_g\| - (2g+1)\delta_g \]
for the hyperelliptic Kawazumi-invariant.

\vspace{1cm}

\noindent Address of the author: \\  \\
Robin de Jong \\
Mathematical Institute \\
University of Leiden \\
PO Box 9512 \\
2300 RA Leiden \\
The Netherlands \\
Email: \verb+rdejong@math.leidenuniv.nl+

\end{document}